\documentclass[11pt,a4paper,reqno]{amsart}
\usepackage{amsthm,amsmath,amsfonts,amssymb,amsxtra,appendix,bookmark,dsfont,latexsym,bm,hyperref,delarray,color,euscript,amsgen,amsbsy,amsopn,amscd,latexsym,mathrsfs}

\usepackage{hyperref}
\usepackage{enumitem}

\newtheorem{theorem}{Theorem}
\numberwithin{theorem}{section}
\newtheorem{lemma}[theorem]{Lemma}
\newtheorem{proposition}[theorem]{Proposition}

\theoremstyle{definition}
\newtheorem{definition}[theorem]{Definition}

\theoremstyle{remark}
\newtheorem{remark}[theorem]{Remark}

\DeclareMathOperator{\dist}{dist}
\DeclareMathOperator{\diam}{diam}

\DeclareMathOperator{\dom}{Dom}
\DeclareMathOperator{\sign}{sign}

\numberwithin{equation}{section}

\begin{document}

\title[Nonlocal elliptic equations with singular nonlinearities]{A  large class of nonlocal elliptic equations with singular nonlinearities}

\author[R. Arora]{Rakesh Arora}
\address[R. Arora]{Department of Mathematical Sciences, Indian Institute of Technology Varanasi (IIT-BHU), Uttar Pradesh-221005, India}
\email{\tt rakesh.mat@iitbhu.ac.in, arora.npde@gmail.com}

\author[P.-T. Nguyen]{Phuoc-Tai Nguyen}
\address[P.-T. Nguyen]{Department of Mathematics and Statistics, Masaryk University, Brno, Czech Republic}
\email{\tt ptnguyen@math.muni.cz}

\author[V.D. R\u adulescu]{Vicen\c tiu D. R\u adulescu}
\address[V.D. R\u adulescu]{Faculty of Applied Mathematics, AGH University of Science and Technology, 30-059 Krak\'ow, Poland \newline
 Department of Mathematics, University of Craiova,  200585 Craiova, Romania \newline
 Simion Stoilow Institute of Mathematics of the Romanian Academy, 21 Calea Grivitei Street, 010702 Bucharest, Romania}
\email{\tt radulescu@inf.ucv.ro}


\begin{abstract}
	 In this work, we address the questions of existence, uniqueness, and boundary behavior of the positive weak-dual solution of equation $\mathbb{L}_\gamma^s u = \mathcal{F}(u)$, posed in a $C^2$ bounded domain $\Omega \subset \mathbb{R}^N$, with appropriate homogeneous boundary or exterior Dirichlet conditions. The operator $\mathbb{L}_\gamma^s$ belongs to a general class of nonlocal operators including typical fractional Laplacians such as restricted fractional Laplacian, censored fractional Laplacian and spectral fractional Laplacian. The nonlinear term $\mathcal{F}(u)$ covers three different amalgamation of  nonlinearities: a purely singular nonlinearity $\mathcal{F}(u) = u^{-q}$ ($q>0$), a singular nonlinearity with a source term $\mathcal{F}(u) = u^{-q} + f(u)$,  and a singular nonlinearity with an absorption term $\mathcal{F}(u) = u^{-q}-g(u)$. Based on a delicate analysis of the Green kernel associated to $\mathbb{L}_\gamma^s$, we develop a new unifying approach that empowered us to construct a theory for equation $\mathbb{L}_\gamma^s u = \mathcal{F}(u)$. In particular, we show the existence of two critical exponents $q^{\ast}_{s, \gamma}$ and $q^{\ast \ast}_{s, \gamma}$ which provides a fairly complete classification of the weak-dual solutions via their  boundary behavior. Various types of nonlocal operators are discussed to exemplify the wide applicability of our theory.
	
\smallskip
	
\noindent\textsc{Key words}: singular nonlinearities; nonlocal elliptic equations; Green function; weak-dual solution; sub-super solution method.

\smallskip
	
\noindent\textsc{2020 Mathematics Subject Classification}: 35A01, 35B33, 35B65, 35J61, 35J75, 47G20.

\end{abstract}


\maketitle

\tableofcontents

\section{Introduction}
The development of nonlocal analysis during the last few decades has been profoundly influenced by the attempts to understand various real-world phenomena from numerous areas such as probability and statistics, finance, mathematical physics, and other scientific disciplines. A great amount of attention has been paid to several aspects of elliptic equations involving different kinds of operators varying from typical fractional Laplacians to a more general integro-differential operator. Lately, the study of elliptic equations involving a wide class of nonlocal operators determined by two-sided estimates of the associated Green kernel has come to the light (see for instance \cite{AbaGomVaz_2019, BonFigVaz_2018,  BonSirVaz_2015, BonVaz_2014, BonVaz_2015, ChaGomVaz_2020, HuNg_source, HuNg_absorption} and references therein).

In this work, we focus on elliptic problems involving a large class of nonlocal operators $\mathbb{L}_\gamma^s$ in the presence of a singular nonlinearity and a source or an absorption term of the form
\begin{equation}\label{eq:mainprob}
    \mathbb{L}_\gamma^s u = \mathcal{F}(u),\quad u>0 \quad \text{in} \ \Omega,
\end{equation}
with homogeneous  boundary or exterior Dirichlet condition
\begin{equation}\label{eq:boundcond}
   u=0 \ \text{on} \ \partial \Omega \ \text{or in} \ \Omega^c  \ \text{if applicable},
\end{equation}
where $\Omega$ is a $C^2$ bounded domain in $\mathbb{R}^N$ ($N \geq 1$), $\mathcal{F}$  is a nonlinear term representing three different kinds of singular nonlinearities and $\mathbb{L}_\gamma^s$ is a local or nonlocal operator. The parameters $\gamma$ and $s$ depict the interior point singularity and the boundary behavior of the  Green function associated to $\mathbb{L}_{\gamma}^s$ (see Section \ref{Main assumptions} for more details).

The study of elliptic or integral equations involving singular terms started in the early sixties by the works of Fulks and Maybee \cite{Fulks.al_1960}, originating from the models of steady-state temperature distribution in an electrically conducting medium. Later on, equations with singular nonlinearities attracted the attention of many researchers. On the one hand, the study of such types of equations is a challenging mathematical problem. On the other hand, they appear in a variety of real-world models. We refer here to \cite{Fulks.al_1960} for applications in the theory of heat conduction in electrically conducting materials, \cite{Nachman.al_1986} in the pseudo-plastic fluids, \cite{Stuart_1974} for Chandrasekhar equations in radiative transfer, and \cite{Diaz.al_1987} in non-Newtonian fluid flows in porous media and heterogeneous catalysts. \vspace{0.1cm}

One of the seminal breakthroughs in this area was the work of Crandall, Rabinowitz, and Tartar \cite{Crandall.al_1977} which majorly provoked the research in the direction of the study of singular nonlinearities. Afterward, a large number of publications has been devoted to investigating a diverse spectrum of issues revolving around local/nonlocal elliptic equations involving the singular nonlinearities (see for example \cite{Diaz.al_1987, Fulks.al_1960, Stuart_1974} and monographs \cite{Ghergu.al_2008, Hernandez.al_2006}).  Let us recall some known results in the literature for both local and nonlocal elliptic equations with singular nonlinearities. \vspace{0.1cm}

In the local case, Crandall, Rabinowitz and Tartar \cite{Crandall.al_1977} studied the singular boundary value problem \eqref{eq:mainprob}--\eqref{eq:boundcond} with $\mathbb{L}_\gamma^s= -\Delta$, {\it i.e.} $s=\gamma=1$, and $\mathcal{F}(u)= u^{-q}$ for $q>0$. By using the classical method of sub and supersolutions on the non-singular approximating problem, they proved the existence and uniqueness results of the classical solution of our original problem. In addition, by exploiting the second-order ODE techniques and localization near the boundary, they showed the existence of a critical exponent which is equal to $1$ and derived the boundary behavior of the solution. Thanks to Stuart \cite{Stuart_1976}, similar results on the existence of solutions were obtained using, this time, an approximation argument with respect to the boundary condition. Actually, both papers \cite{Crandall.al_1977, Stuart_1976} provide results for more general differential operators
with smooth coefficients, not necessarily in divergence form, and for non-monotone
nonlinearities as well. The same model of elliptic equations with purely singular nonlinearity was considered by Lazer and McKenna \cite{Lazer.al_1991} in which they simplified the proof of the boundary behavior of classical solutions by constructing appropriate sub and super solutions. For similar works concerning the local elliptic or integral equations with a  purely singular nonlinearity, we refer to \cite{Zheng.al_2004, Diaz.al_1987, Gomes_1986, delpino_1992, Stuart_1974} and for singular nonlinearity with source terms or absorption terms, we refer to \cite{Coc.al_1989, Stuart_1976, Can.al_2004, haitao_2003, Hirano.al_2004} with no intent to furnish an exhaustive list. \vspace{0.1cm}

Concerning the nonlocal case, there have been very few available works on singular equations in the literature, mainly dealing with the well-known fractional Laplace operator (see for instance \cite{adi_gia_santra_2018, Arora_Gia_Goel_Sree_2020, Arora_Gia_Warn_2021, Arora_Rad_2021, Barrios.al_2015, Giacomoni.al_2017}). In \cite{adi_gia_santra_2018}, Adimurthi, Giacomoni and Santra studied problem \eqref{eq:mainprob}--\eqref{eq:boundcond} with $\mathbb{L}_{\gamma}^s$ being the fractional Laplace operator, namely $\mathbb{L}_\gamma^s = (-\Delta)^s$, and  $\mathcal{F}(u)= u^{-q}$ with $q>0$. They discussed the existence and uniqueness of the classical solution using the approximation method and comparison principle. Moreover, by employing the integral representation via the Green function and maximum principle, they showed the existence of a critical exponent which is equal to $1$ and obtained the sharp boundary behavior of the weak solution. Very lately, in \cite{Arora_Rad_2021}, Arora and R\u adulescu studied the problem \eqref{eq:mainprob}--\eqref{eq:boundcond} with $\mathbb{L}_\gamma^s = (-\Delta) + (-\Delta)^s$, and  $\mathcal{F}(u)= u^{-q}$ with $q>0$. They established the existence, uniqueness and regularities properties of the weak solution by deriving uniform apriori estimates
and using the method of approximation.  When  $\mathbb{L}_\gamma^s = (-\Delta)^s$ and $\mathcal{F}(u)= u^{-q} + \lambda f(u)$ for $q>0$, $\lambda >0$ and $f$ satisfies subcritical, critical growth or exponential nonlinearities, we refer to the work \cite{Arora_Gia_Goel_Sree_2020, Barrios.al_2015, Giacomoni.al_2017}. For further issues on nonlocal singular problems, the interested reader can consult the
bibliographic references in  \cite{Arora_Gia_Warn_2021, Arora_Rad_2021, Mukherjee.al_2016}.

Recently, elliptic and parabolic equations involving a large class of operators characterized by its Green function has been studied in a series of works \cite{AbaGomVaz_2019, Arora_Nguyen_2022, BonFigVaz_2018, BonSirVaz_2015, BonVaz_2014, BonVaz_2015, ChaGomVaz_2020, Gomez_Vaz_2019} from the PDE point of view and in \cite{Kim_Song_Vondra_2020, Kim_Song_Vondrac_2020} from the probabilistic point of view. An important contribution in this research direction is the work of Bonforte and V\'azquez \cite{BonSirVaz_2015} where the concept of weak-dual solutions was introduced to investigate the existence and uniqueness of solutions of nonlinear diffusion evolution equations. Later on, Bonforte, Figalli and V\'azquez \cite{BonFigVaz_2018} studied the qualitative properties of weak-dual solutions of semilinear elliptic equations \eqref{eq:mainprob}-\eqref{eq:boundcond} with $\mathcal{F}(u) = u^p$, $0 < p < 1$. More precisely, they proved that the weak-dual solution $u$ has the following sharp boundary estimates
\[
 \begin{aligned}
u(x) \asymp \left\{
	\begin{array}{ll}
	  \delta(x)^{\min\{\gamma, \frac{2s}{1-p}\}} &  \text{ if } 2s + p \gamma \neq \gamma,
	\vspace{0.1cm}\\
\delta(x)^\gamma \ln^{\frac{\gamma}{2s}}\left(\frac{d_\Omega}{\delta(x)}\right)&  \text{ if } 2s + p \gamma = \gamma,
	\end{array}
	\right.
 \end{aligned}
\]
where $d_{\Omega}=\textrm{diam}(\Omega)$ and $\delta(x)=\dist(x,\mathbb{R}^N \setminus \Omega)$.
Thereafter, Abatangelo, G\'omez-Castro and V\'azquez in \cite{AbaGomVaz_2019} and Chan, G\'omez-Castro and V\'azquez in \cite{ChaGomVaz_2020} proved some more results related to the boundary behavior of solutions and eigenvalue problems for equations involving general nonlocal operators. Very lately, Huynh and Nguyen studied the semilinear elliptic equations involving superlinear source terms  $\mathcal{F}(u) = u^p + \mu$ (see \cite{HuNg_source}) and absorption terms $\mathcal{F}(u) = \mu - g(u)$ (see \cite{HuNg_absorption}), where $\mu$ is a Radon measure on $\Omega$, $p>1$ and $g: \mathbb{R} \to \mathbb{R}$ is non-decreasing continuous function.
By developing a set of unifying techniques based on a fine analysis of the Green function, they established the existence/non-existence and uniqueness/multiplicity and qualitative properties of weak dual solutions.

Apart from these publications, we are unaware of works on the existence and boundary behavior of the solution of elliptic problem driven by different kinds of local or nonlocal operators, whilst the results for equations involving singular nonlinearities seem to be completely missing. The interaction between the properties of the operator and the presence of the singular nonlinearities may lead to
disclose new types of difficulties in the analysis and requires a novel approach.
\vspace{0.1cm}\\
\textbf{Objectives of this paper.} Motivated from the above mentioned works, we aim to set up a general framework and to develop unifying techniques which enable us to study the existence, uniqueness and the boundary behavior of weak-dual solution to problem \eqref{eq:mainprob}--\eqref{eq:boundcond} for various types of operators $\mathbb{L}_{\gamma}^s$ and different types of singular nonlinearities $\mathcal{F}$. The class of operators $\mathbb{L}_\gamma^s$ under considerations is assumed to satisfy a set of mild hypotheses  and cover famous nonlocal operators. The elementary assumptions on $\mathbb{L}_\gamma^s$ is described in terms of two-sided estimates of the associated Green function and new convexity inequality. The main results of this article are novel for various case of operators, cover and extend the aforementioned works regarding problems with singular nonlinearities in the literature, and complement the works of Bonforte, Figalli and V\'azquez in \cite{BonFigVaz_2018}, and Huynh and Nguyen in \cite{HuNg_source, HuNg_absorption}.
\vspace{0.1cm}\\
\textbf{Outline of the paper.} The rest of the paper is organized as follows. In Section \ref{def-assump}, we
give the definition of function spaces and present the main assumptions on the operator $\mathbb{L}_{\gamma}^s$. We also give a list of examples of operators to which our theory can be applied. In Section \ref{Main-prob-results}, we introduce the main problems, the notion of solutions and statement of main results. In section \ref{prelim-kato-ineq}, we recall some crucial preliminary results and prove the Kato's inequality which plays the role of comparison principle in our general setting. Section \ref{PS-problem} is devoted to study the purely singular problem \eqref{sing:pure} (see below). First, we establish sharp estimates regarding the action of the Green operator on negative powers of distance function perturbed with logarithmic growth. This in turn enables to prove the existence of a weak-dual solution via approximation method and Schauder fixed point theorem. As an application of our new Kato type inequality, we show the uniqueness and boundary behavior of weak-dual solutions. In section \ref{source:absorption}, we study the semilinear elliptic equations \ref{sing:f(v)} and \ref{sing+absorption} (see below) involving a singular nonlinearity and a source/absorption term and prove the existence of a weak-dual solution via sub-super solution method. In section \ref{example:singpotential}, we discuss the applicability of our results on the Laplace operator perturbed by Hardy potential.
\section{Definitions,  assumptions and examples}\label{def-assump}
\subsection{Functional spaces}
Throughout the paper, the symbol $c,C,c_i,C_i$ ($i=1,2,\dots$) will be used to denote the constants which can be calculated or estimated through the data. If necessary, we will write $C=C(a,b,\dots)$ to emphasize the dependence of $C$ on $a,b$. We fix some shorthand notations: for two functions $f,g$, we write $f \lesssim g$ if there exists a $C>0$ such that $f \leq C g$ and $f \asymp g$ if $f \lesssim g$ and $g \lesssim f$. We also write $a \wedge b := \min\{a,b\}$ and $a \vee b := \max\{a,b\}$. Finally, we assume $\Omega$ is a $C^2$ bounded domain in $\mathbb{R}^N$ ($N \geq 2$) and denote $\delta(x)=\dist(x,\mathbb{R}^N \setminus \Omega)$. We also denote $d_{\Omega}=2\textrm{diam}(\Omega)$. For a Borel set $A \subset \mathbb{R}^N$, ${\bf 1}_A$ denotes the indicator function of $A$.

For a given non-negative function $\rho$ and $p \in [1, \infty)$, the weighted Lebesgue space $L^{p}(\Omega, \rho)$ is defined as
$$L^{p}(\Omega, \rho):= \left\{f: \Omega \to \mathbb{R} \text{ Borel measurable such that}  \int_\Omega |f|^p \rho ~dx < \infty \right\}$$ endowed with the norm $$\|f\|_{L^p(\Omega, \rho)} = \left(\int_{\Omega} |f|^p \rho ~dx\right)^\frac{1}{p}$$ and
$$L^{p}_0(\Omega, \rho):= \left\{f \in L^{p}(\Omega, \rho): f=0 \ \mbox{ on } \partial\Omega\, \ \mbox{or in $\Omega^c$ if applicable} \right\}.$$
We also denote by $\rho L^\infty(\Omega)$ the space
$$ \rho L^\infty(\Omega):= \{f: \Omega \to \mathbb{R} \ | \ \text{there exists} \ v \in L^\infty(\Omega) \ \text{such that} \ f= \rho v\}.$$

For $s \in (0,1)$, the fractional Sobolev space $H^s(\Omega)$ is defined by
\[
H^s(\Omega):= \{ u \in L^2(\Omega): [u]_{s, \Omega}:= \int_{\Omega} \int_{\Omega} \frac{|u(x) - u(y)|^2}{|x-y|^{N+2s}} ~dx ~dy < + \infty\},
\]
which is a Hilbert space equipped with the inner product
\[
\langle u, v\rangle_{H^s(\Omega)} := \int_{\Omega} uv ~dx + \int_{\Omega} \int_{\Omega} \frac{(u(x) - u(y))(v(x)-v(y))}{|x-y|^{N+2s}} ~dx ~dy, \quad u, v \in H^s(\Omega).
\]
The space $H_0^s(\Omega)$ is defined as the closure of $C_c^\infty(\Omega)$ with respect to the norm
\[
\|u\|_{H^s_{0}(\Omega)} := \sqrt{\langle u, u \rangle_{H^s(\Omega)}}.
\]
 Next, we define $H_{00}^s(\Omega)$ given by
 \[
 H_{00}^s(\Omega):= \left\{u \in H^s(\Omega) : \frac{u}{\delta^s} \in L^2(\Omega)\right\}.
 \]
 It can be seen that the space $H_{00}^s(\Omega)$ endowed with the norm
 \[
 \|u\|_{H^s_{00}(\Omega)} :=  \left(\int_{\Omega} \left(1+ \frac{1}{\delta^{2s}} \right)|u|^2 ~dx + \int_{\Omega} \int_{\Omega} \frac{|u(x) - u(y)|^2}{|x-y|^{N+2s}} ~dx ~dy \right)^{\frac{1}{2}}
 \]
 is a Banach space. Speaking roughly, the space $H_{00}^s(\Omega)$ is the space of functions in $H^s(\Omega)$ satisfying the Hardy's inequality. By  \cite[Subsection 8.10]{Bhattacharyya-12}, there holds
 \[
 \begin{aligned}
H_{00}^s(\Omega) =  \left\{
	\begin{array}{ll}
	  H_{0}^s(\Omega)&  \text{ if } 0< s< \frac{1}{2},
	\vspace{0.1cm}\\
	H_{00}^\frac{1}{2}(\Omega) \subset H_0^\frac{1}{2}(\Omega) &  \text{ if } s=\frac{1}{2},\\
	 H_0^s(\Omega) &  \text{ if } \frac{1}{2} < s < 1.
	\end{array}
	\right.
 \end{aligned}
 \]
In fact, $H_{00}^s(\Omega)$ is the space of functions in $H^s(\mathbb{R}^N)$ supported in $\Omega$, or equivalently, the trivial
extension of functions in $H_{00}^s(\Omega)$ belongs to $H^s(\mathbb{R}^N)$ (see \cite[Lemma 1.3.2.6]{Grisvard-2011}). Furthermore, $C_c^\infty(\Omega)$ is a dense subset of $H_{00}^s(\Omega)$.  The strict inclusion $H_{00}^\frac{1}{2}(\Omega) \subset H_0^\frac{1}{2}(\Omega)$ holds since $1 \in H_0^\frac{1}{2}(\Omega)$ but
$1 \not\in H_{00}^\frac{1}{2}(\Omega).$
\subsection{Main assumptions}\label{Main assumptions}
We consider a general family of linear operators $\mathbb{L}_\gamma^s$ indexed on two parameters depicting the interior point singularity and boundary behavior of the Green kernel. The operator $\mathbb{L}_\gamma^s$ includes the three most typical fractional Laplace operators, as well as the classical Laplace operator.

\textit{In the sequel, for the sake of simplicity, we write $\mathbb{L}$ for $\mathbb{L}_\gamma^s$.}

\noindent \textbf{Assumptions on $\mathbb{L}$.} We impose the following assumptions on the operator $\mathbb{L}$.

\begin{enumerate}[resume,label=(L\arabic{enumi}),ref=L\arabic{enumi}]
	\item \label{L1} $\mathbb{L}: C^{\infty}_c(\Omega) \subset L^2(\Omega) \to L^2(\Omega)$ is a positive, symmetric operator.
\end{enumerate}
Under the above assumption, we infer from the standard spectral theory that $\mathbb{L}$ admits a positive, self-adjoint extension $\widetilde{\mathbb{L}}$, which is a Friedrich's extension of $\mathbb{L}.$ Furthermore, $\mathbb{H}(\Omega) := \dom(\widetilde{\mathbb{L}}^\frac{1}{2})$ is a Hilbert space equipped with the inner product
\[
(u,v) \mapsto \langle u,v \rangle_{L^2(\Omega)} + \mathbb{B}(u,v) , \quad u, v \in \mathbb{H}(\Omega),
\]
where the $\mathbb{B}$ is a bilinear form defined as
\[
\mathbb{B}(u,v):= \langle \widetilde{\mathbb{L}}^\frac{1}{2} u, \widetilde{\mathbb{L}}^\frac{1}{2} v \rangle_{L^2(\Omega)}, \quad u, v \in \mathbb{H}(\Omega).
\]
The inner product induces the following norm on $\mathbb{H}(\Omega)$
\begin{equation} \label{norm0}
 \sqrt{\|u\|^2_{L^2(\Omega)} + \mathbb{B}(u,u)}, \quad u \in \mathbb{H}(\Omega).
\end{equation}
\textit{Hereafter, without any confusion we use the same notation $\mathbb{L}$ to denote the extension $\widetilde{\mathbb{L}}$.}

\begin{enumerate}[resume,label=(L\arabic{enumi}),ref=L\arabic{enumi}]
	\item \label{L2}
    There exists a constant $\Lambda >0$ such that
    $$
    \|u\|^2_{L^2(\Omega)} \leq \Lambda \langle \mathbb{L} u, u\rangle_{L^2(\Omega)},  \quad \text{for all} \ u \in C_c^\infty(\Omega).
    $$
\end{enumerate}
Under the assumption \eqref{L2} and the fact that $C_c^\infty(\Omega)$ is dense in $\mathbb{H}(\Omega)$, we have
\[
\|u\|_{L^2(\Omega)}^2 \leq \Lambda  \langle \mathbb{L}^\frac{1}{2} u, \mathbb{L}^\frac{1}{2} u \rangle_{L^2(\Omega)} = \Lambda \mathbb{B}(u,u), \quad \text{for all} \ u \in \mathbb{H}(\Omega),
\]
which further implies the norm in \eqref{norm0} and the norm $\|u\|_{\mathbb{H}(\Omega)} := \sqrt{\mathbb{B}(u,u)}$ are equivalent  on $\mathbb{H}(\Omega).$ The norm $\|\cdot\|_{\mathbb{H}(\Omega)}$ is induced by the inner product
\begin{equation} \label{innerH}
\langle u, v\rangle_{\mathbb{H}(\Omega)} := \mathbb{B}(u,v), \quad \text{for} \ u,v \in \mathbb{H}(\Omega).
\end{equation}

We additional assume that
\begin{enumerate}[resume,label=(L\arabic{enumi}),ref=L\arabic{enumi}]
	\item \label{L3}
	 $\mathbb{H}(\Omega) = H^s_{00}(\Omega)$.
\end{enumerate}
This assumption is indicated and motivated by the fact that it is fulfilled in the case of well-know fractional Laplace operators.

Next we assume that the following convexity inequality holds.
\begin{enumerate}[resume,label=(L\arabic{enumi}),ref=L\arabic{enumi}]
	\item \label{L4}
	For $u,v\in \mathbb{H}(\Omega) \cap L^\infty(\Omega)$ with $v \geq 0$ and smooth convex function $p \in C^{1,1}(\mathbb{R})$ with $p(0)=p'(0)=0$, there holds
	\[
	\left\langle p(u), v \right\rangle_{\mathbb{H}(\Omega)} \leq \left\langle  u,  p'(u)  v \right\rangle_{\mathbb{H}(\Omega)}.
	\]
\setcounter{enumi}{0}
\end{enumerate}

\noindent \textbf{Assumptions on the inverse of $\mathbb{L}$.} We also require the existence of the inverse operator of $\mathbb{L}$.

\begin{enumerate}[resume,label=(G\arabic{enumi}),ref=G\arabic{enumi}]
	\item \label{G1}
There exists an operator $\mathbb{G}^{\Omega}$ such that for every $f \in C^{\infty}_c(\Omega)$,  one has
	\begin{equation}\label{eq:homolinear}
		\mathbb{L} [\mathbb{G}^{\Omega} [f]] = f \quad \text{ a.e. in } \Omega
	\end{equation}
	In other words,  $\mathbb{G}^{\Omega}$ is a right inverse of $\mathbb{L}$  and for every $f \in C^{\infty}_c(\Omega)$, one has $\mathbb{G}^{\Omega}[f] \in \dom(\mathbb{L}) \subset \mathbb{H}(\Omega)$.
\end{enumerate}
	
\begin{enumerate}[resume,label=(G\arabic{enumi}),ref=G\arabic{enumi}]
	\item \label{G2}
	 The operator $\mathbb{G}^{\Omega}$ admits a Green kernel $G^{\Omega}$, namely
	$$
	\mathbb{G}^{\Omega}[f](x) = \int_{\Omega} G^{\Omega}(x,y)f(y)~dy,\quad x \in \Omega,
	$$
where $G^{\Omega}: \dom(G^{\Omega}) := \Omega \times \Omega \setminus \{(x,x): x \in \Omega\} \to \mathbb{R}^+$ is Borel measurable, symmetric and satisfies the following two-sided estimate
	\begin{equation} \label{G-est}
		G^{\Omega} (x,y) \asymp \dfrac{1}{|x-y|^{N-2s}} \left( \dfrac{\delta(x)}{|x-y|} \wedge 1 \right)^{\gamma} \left( \dfrac{\delta(y)}{|x-y|} \wedge 1 \right)^{\gamma}, \; x,y \in \dom(G^{\Omega}),
	\end{equation}
	with $s,\gamma \in (0,1]$ and $N>2s$.
\end{enumerate}

Finally, we require the continuity property of $G^\Omega$
\begin{enumerate}[resume,label=(G\arabic{enumi}),ref=G\arabic{enumi}]
	\item \label{G3}
For $x \in \Omega$ and any sequence ${x_n} \subset \Omega$ such that $x_n \to x$ as $n \to \infty$, we have $G^{\Omega}(x_n,y) = G^{\Omega}(x, y)$ for a.e. $x \neq y \in \Omega$.
\setcounter{enumi}{0}
\end{enumerate}
\subsection{Examples} \label{subsec:examples} At first glance, it seems that the set of assumptions in Subsection \ref{Main assumptions} is quite restrictive, however there are several local and nonlocal operators satisfying these assumptions, as shown below. It was proved in \cite{HuNg_absorption} that assumptions \eqref{L1}--\eqref{L3} and \eqref{G1}--\eqref{G2} are fulfilled for the following operators. Therefore, it is sufficient to verify only \eqref{L4} and \eqref{G3}.
\vspace{0.1cm}\\
\textbf{The restricted fractional Laplacian (RFL).}
A famous example of nonlocal operator satisfying the set of assumptions in Subsection \ref{Main assumptions} is the restricted fractional Laplacian $\mathbb{L} = (-\Delta)^s_{\text{RFL}}$ defined, for any $s \in (0,1)$, by
$$(-\Delta)^s_{\text{RFL}} u(x) := a_{N,s}\ P.V. \int_{\mathbb{R}^N} \frac{u(x)-u(y)}{|x-y|^{N+2s}} ~dy, \quad x \in \Omega,$$
restricted to the functions that vanishes outside $\Omega$. Here the abbreviation P.V. stands for  ``the principal value sense". This operator corresponds to the $s$-power of classical Laplace operator and has been intensively studied in the literature; see, e.g.  \cite{CafSil,Sil,Nezza_Palatucci_Valdi_2012,Aba_2015}.

We consider the following bilinear form
\begin{equation}\label{bilinear-form}
	\mathcal{B}(u,v) :=  \frac{1}{2} \int_{\Omega} \int_{\Omega} J(x,y) (u(x) - u(y))(v(x) - v(y)) ~dx ~dy + \int_{\Omega} B(x) u(x) v(x) ~dx
\end{equation}
with
\[
J(x,y) := \frac{a_{N,s}}{|x-y|^{N+2s}} \quad \text{and} \quad B(x) := a_{N,s} \int_{\mathbb{R}^N \setminus\Omega} \frac{1}{|x-y|^{N+2s}} ~dy, \quad x, y \in  \Omega, \ x \neq y
\]
on the domain
\[
\mathcal{D}(\mathcal{B}) := \{ u \in L^2(\Omega): \mathcal{B}(u, u) < + \infty\}.
\]
By \eqref{L1} and \cite[Proposition 5.1]{HuNg_absorption}, we deduce $C_c^\infty(\Omega) \subset \mathcal{D}(\mathcal{B})$ and $\mathcal{D}(\mathcal{B})$ is a complete space with respect to the norm
\[
\|u\|_{\mathcal{D}(\mathcal{B})} := \left( \|u\|_{L^2(\Omega)}^2 + \mathcal{B}(u,u) \right)^\frac{1}{2}.
\]
Moreover,
\[
\langle u,v \rangle_{\mathbb{H}(\Omega)} = \mathbb{B}(u,v) = \langle \widetilde{\mathbb{L}}^\frac{1}{2} u, \widetilde{\mathbb{L}}^\frac{1}{2} u \rangle_{L^2(\Omega)} = \mathcal{B}(u,v) \quad \text{for all} \ u, v \in \mathbb{H}(\Omega).
\]

To show that $(-\Delta)^s_{\text{RFL}}$ satisfies assumption \eqref{L4}, let us take a convex function $p \in C^{1,1}(\mathbb{R})$ such that $p(0) = p'(0) =0$ and $u, v \in \mathbb{H}(\Omega) \cap L^\infty(\Omega)$ with $v \geq 0$. By Lemma \ref{lip-prop}, we obtain
$p(u), p'(u) v \in \mathbb{H}(\Omega)$. The convexity of the function $p$ implies
\begin{equation}\label{convex-ineq}
   (p(a) -p(b))  (A-B) \leq  (a-b)(Ap'(a)-Bp'(b)) \quad \text{and} \quad p(a) \leq p'(a) a,
\end{equation}
for every $a, b \in \mathbb{R}$ and $A, B \geq 0$. By taking $a=u(x)$, $b = u(y)$, $A= v(x)$ and $B=v(y)$ in \eqref{convex-ineq}, we derive
\[
\langle p(u),v \rangle_{\mathbb{H}(\Omega)}
=  \mathcal{B}(p(u), v) \leq \mathcal{B}(u, p'(u) v) = \langle u,p'(u) v \rangle_{\mathbb{H}(\Omega)}.
\]
Hence, \eqref{L4} holds. Finally, \eqref{G3} follows from \cite[Page 467]{Chen-song_1998}.
\vspace{0.1cm}\\
\textbf{The censored fractional Laplacian (CFL).} The censored fractional Laplacian is defined for $s > \frac{1}{2}$ by $$(-\Delta)^s_{\text{CFL}} u(x) := b_{N,s}\ P.V. \int_{\Omega} \frac{u(x)-u(y)}{|x-y|^{N+2s}} ~dy, \quad x \in \Omega.$$
Stochastically speaking, the CFL generates a censored $2s$-stable process restricted in $\Omega$ and killed upon hitting the boundary of $\Omega$ (for more details see \cite{Bogdan.al_2003, Chen.al_2002, Chen.al_2010}).

Assumption \eqref{L4} holds true by repeating the same arguments as in the case of RFL by considering the bilinear form $\mathcal{B}$ as in \eqref{bilinear-form}
with
\[
J(x,y) = \frac{b_{N,s}}{|x-y|^{N+2s}} \quad \text{and} \quad B(x) = 0, \quad x, y \in \Omega, \ x \neq y.
\]
Finally, \eqref{G3} follows from \cite[Theorem 1.1]{Chen.al_2002}.
\vspace{0.1cm}\\
\textbf{The spectral fractional Laplacian (SFL).} The spectral fractional Laplacian is defined for $s \in (0,1)$ by
$$(-\Delta)^s_{\text{SFL}} u(x) := C(N,s)\ P.V. \int_{\mathbb{R}^N} (u(x)-u(y)) J_s(x,y) ~dy + I_s(x) u(x), \quad x \in \Omega,$$
where
$$J_s(x,y):= \frac{s}{\Gamma (1-s)} \int_0^\infty  \frac{p_\Omega(t,x,y)}{t^{1+s}} ~dt , \quad x,y \in \Omega$$
and $$I_s(x):= \frac{s}{\Gamma(1-s)} \int_0^\infty \left(1- \int_{\Omega} p_\Omega(t,x,y) ~dy \right) \frac{1}{t^{1+s}}~dt  \asymp \frac{1}{\delta(x)^{2s}}, \quad x \in \Omega.$$
Here $p_\Omega$ is the heat kernel of the Laplace operator $(-\Delta)$ in $\Omega.$ The SFL generates a subordinate killed Brownian motion (a Brownian motion killed upon hitting the boundary of $\Omega$) and then treated as an $s$-stable process in $(0, \infty)$ (for more details see \cite{Bogdan.al_2009, Song.al_2003}).

Assumption \eqref{L4} follows by repeating the same arguments as in the case of RFL with $J(x,y) = J_s(x,y)$ and $B(x) = I_s(x)$ in \eqref{bilinear-form}. Finally, assumption \eqref{G3} is satisfied
due to the proof of Lemma $14$ in \cite[Page 7]{Abatangelo.al_2017}.
\vspace{0.1cm}\\
\textbf{Sum of two RFLs.} Let $0 < \beta < \alpha <1.$ The sum of $\alpha$-RFL and $\beta$-RFL is defined as
\[
\mathscr{L}_{\alpha,\beta} \, u(x) := (-\Delta)^\alpha_{\text{RFL}} u(x) + (-\Delta)^\beta_{\text{RFL}} u(x) = P.V. \int_{\mathbb{R}^N} J_{\alpha, \beta}(x,y) (u(x)-u(y)) ~dy,
\]
where
\[
J_{\alpha, \beta}(x,y) :=  \frac{a_{N,\alpha}}{|x-y|^{N+2\alpha}} + \frac{a_{N,\beta}}{|x-y|^{N+2\beta}}.
\]
This type of operators has been studied in \cite{Chen_Kim_Song_2012}. By proceeding as in case of RFL with minor modification, we can easily show that  \eqref{L4} is satisfied. Finally, to show \eqref{G3}, we use the following integral representation of the Green kernel given by continuous transition density function (or heat kernel) $p_{\Omega}^{\alpha, \beta}$
\[
G^\Omega(x,y) = \int_0^\infty p_{\Omega}^{\alpha, \beta}(t,x,y) ~dt.
\]
Let $(x,y) \in \dom(G^\Omega)$ and assume that $|x-y|>2\eta>0$. Let $\{(x_n, y_n)\}$ be a sequence in $\dom(G^\Omega)$ converging to $(x,y)$. For $n$ large enough, we have $|x_n -y_n| > \eta.$ From \cite[Theorem 1]{Chen_Kim_Song_2012}, there exists a uniformly dominated integrable function $g_\Omega$ satisfying
\[
\begin{split}
& p_{\Omega}^{\alpha, \beta}(t,x_n,y_n) \\
& \quad \leq C_1 \left\{
	\begin{array}{ll}
	  \left(1 \wedge \frac{\delta(x_n)^\alpha}{\sqrt{t}}\right)  \left(1 \wedge \frac{\delta(y_n)^\alpha}{\sqrt{t}}\right)  \left(t^{-\frac{N}{2 \alpha}} \wedge \left( \frac{t}{|x_n-y_n|^{N+2\alpha}} +  \frac{t}{|x_n-y_n|^{N+2\beta}} \right) \right)&  \text{ if } t \leq T_0,
	\vspace{0.1cm}\\
e^{-\lambda_1 t} \delta(x_n)^\alpha \delta(y_n)^\alpha &  \text{ if } t \geq T_0,
	\end{array}
	\right.\\
	& \quad \leq g_\Omega(t) := C_2 \left\{
	\begin{array}{ll}
	  t \eta^{-N-2\alpha} &  \text{ if } t \leq T_0,
	\vspace{0.1cm}\\
e^{-\lambda_1 t} &  \text{ if } t \geq T_0,
	\end{array}
	\right.
 \end{split}
\]
where $C_i=C_i(T_0, \alpha, \beta, \Omega)$, $i=1,2$, and $\lambda_1$ is the smallest eigenvalue of $\mathscr{L}_{\alpha,\beta}$. This means that $0 \leq p_{\Omega}^{\alpha, \beta}(\cdot,x_n,y_n) \leq g_{\Omega} \in L^1((0,+\infty))$. By the continuity of the heat kernel $p_{\Omega}^{\alpha, \beta}$ (see \cite[Page 4]{Chen_Kim_Song_2012}), we have $p_{\Omega}^{\alpha,\beta}(t,x_n,y_n) \to p_{\Omega}^{\alpha,\beta}(t,x,y)$ as $n \to \infty$. Therefore, by employing the Lebesgue dominated convergence theorem, we obtain
$$
\lim_{n \to \infty}G^{\Omega}(x_n,y_n) = \lim_{n \to \infty}\int_0^\infty p_{\Omega}^{\alpha, \beta}(t,x_n,y_n) ~dt = \int_0^\infty p_{\Omega}^{\alpha, \beta}(t,x,y) ~dt =  G^{\Omega}(x,y).
$$
This implies the joint continuity of $G^\Omega$, whence \eqref{G3} follows.
\vspace{0.1cm}\\
\textbf{An interpolation of the RFL and the SFL.} For $\sigma_1, \sigma_2 \in (0, 1]$, we consider the $\sigma_2$-spectral decomposition of the $\sigma_1$-RFL in the domain $\Omega$, namely $[(-\Delta)^{\sigma_1   }_{\text{RFL}}]^{\sigma_2}_{\text{SFL}}$. This type of interpolation
operator has been recently studied in \cite{Kim_Song_Vondra_2020} by a probabilistic approach.

Let $X$ be a rotationally invariant $\sigma_1$-stable process in $\mathbb{R}^N$ and $X^\Omega$ be the subprocess of $X$ that is killed upon existing $\Omega.$ Then, we subordinate $X^\Omega$ by an independent $\sigma_2$-stable subordinator $Y$ to obtain a process $Z^\Omega$, {\it i.e.} $(Z^\Omega)_t = (X^\Omega)_{Y_t}.$ Let $(S_t^\Omega)$ be the semigroup of $Z^\Omega$ whose infinitesimal generator can   be written as
\[
\mathcal{L}_{\sigma_, \sigma_2} := [(-\Delta)^{\sigma_1   }_{\text{RFL}}]^{\sigma_2}_{\text{SFL}}.
\]
In particular, from \cite[(2.20)  and  (3.4)]{Kim_Song_Vondra_2020}, we have that, for $u \in C_c^\infty(\Omega)$,
\[
\mathcal{L}_{\sigma_, \sigma_2} u(x) = \int_{\Omega} (u(x)-u(y)) J_{\sigma_1,\sigma_2}(x,y) ~dy + I_{\sigma_1, \sigma_2}(x) u(x), \quad x \in \Omega
\]
where
\begin{equation}
    \begin{split}
        &J_{\sigma_1, \sigma_2}(x,y):= \int_0^\infty  p_\Omega^{\sigma_, \sigma_2}(t,x,y) ~\nu(dt) , \quad x,y \in \Omega,\\
        & I_{\sigma_1, \sigma_2}(x):= \int_0^\infty \left(1- \int_{\Omega} p_\Omega^{\sigma_1, \sigma_2}(t,x,y) ~dy \right) ~\nu(dt)  \quad x \in \Omega.
    \end{split}
\end{equation}
In the above definition,
\[
\nu(dt) = \frac{\sigma_2}{\Gamma(1-\sigma_2)} t^{-\sigma_2-1} dt
\]
is a L\'evy measure of the $\sigma_2$-stable subordinator $Y$ and $p_\Omega^{\sigma_1, \sigma_2}$ denotes the heat kernel of the $\sigma_1$-RFL in the domain $\Omega$.

Assumption \eqref{L4} follows by repeating the arguments as in the case of RFL with $J(x,y) = J_{\sigma_1, \sigma_2}(x,y)$ and $B(x) = I_{\sigma_1, \sigma_2}(x)$ in \eqref{bilinear-form}. Finally, by employing an analogous argument as in the case of sum of two RFLs, together with the continuity of the heat kernel $p_{\Omega}^{\sigma_1, \sigma_2}$ \cite[Page 17]{Kim_Song_Vondrac_2020} and exit time estimates in \cite[Page 26, (6.2) and (6.3)]{Kim_Song_Vondra_2020}, we obtain the continuity of Green function in \eqref{G3}.
\vspace{0.1cm}\\
\textbf{Restricted relativistic Schr\"odinger operators.} We consider a class of relativistic Schr\"odinger operators with mass $m > 0$ of the form
\[
L^s_m := (-\Delta + m^2 {\mathds I})^s - m^{2s} {\mathds I},
\]
restricted to the class of functions that are zero outside $\Omega$, where $s \in (0,1)$ and ${\mathds I}$ denotes the identity operator.
Alternatively, by \cite[(1.3) and (6.7)]{Fall_Felli_2015}, the operator $L_m^s$ can be expressed as
\[
L_m^s u(x) = c_{N,s} m^{\frac{N+2s}{2}} P.V. \int_{\mathbb{R}^N} \frac{u(x) - u(y)}{|x-y|^\frac{N+2s}{2}} K_{\frac{N+2s}{2}}(m|x-y|) ~dy, \quad x \in \Omega,
\]
where, for $\nu >0$,
\[
K_{\nu}(r) \asymp \frac{\Gamma(\nu)}{2} \left(\frac{r}{2}\right)^{-\nu}, \quad r>0.
\]
and $\Gamma$ is the usual Gamma function. Due to the significant appearance of this operator in many research areas of mathematics and physics, it has attracted  attention of numerous mathematicians; see, e.g.,   \cite{Ryznar_2003,Fall_Felli_2015,Amb_2016}.  Assumption \eqref{L4} is satisfied  by using a similar argument as in the case of RFL with
\[
J(x,y) =  c_{N,s} \frac{K_{\frac{N+2s}{2}}(m|x-y|)}{|x-y|^{\frac{N+2s}{2}}} m^{\frac{N+2s}{2}}, \quad x, y \in \Omega , \ x \neq y\]
and
\[
B(x) = c_{N,s} m^{\frac{N+2s}{2}} \int_{\mathbb{R}^N \setminus \Omega} \frac{K_{\frac{N+2s}{2}}(m|x-y|)}{|x-y|^{\frac{N+2s}{2}}}, \ x \in \Omega,
\]
in the bilinear form \eqref{bilinear-form}. Finally, the continuity of Green function in \eqref{G3} follows from \cite[Theorem 4.6]{Kim_Lee_2007}.

\section{Description of main problems and results}\label{Main-prob-results}
In this section, we study the following nonlocal problem involving the operator $\mathbb{L}$ and  purely singular nonlinearities of the form
\begin{equation}  \label{sing:pure} \tag{$P_s$}
 \left\{
\begin{aligned}
 \mathbb{L} u & = \frac{1}{u^q} \quad && \text{ in } \, \Omega, \\
 u & > 0 \quad && \text{ in } \, \Omega, \\
 u &= 0 \quad && \text{ on } \partial \Omega \text{ or in } \Omega^c \text{ if applicable},
\end{aligned}
\right.
\end{equation}
where $\mathbb{L}$ satisfies the set of assumptions \eqref{L1}-\eqref{L4} and \eqref{G1}-\eqref{G3}, and $q>0$. The notion of weak-dual solutions is given below.
\begin{definition}\label{def:weak-dual}
	A positive function $u$  is said to be a weak-dual solution of the problem \eqref{sing:pure} if
	\begin{equation}\label{eq:notion0}
	u \in L^{1}_0(\Omega, \delta^\gamma),
	\quad u^{-q} \in L^{1}(\Omega, \delta^\gamma)
	\end{equation}
	and
	\begin{equation} \label{eq:notionsol}
	 \int_{\Omega} u \xi ~dx = \int_{\Omega} \frac{1}{u^q} \mathbb{G}^\Omega[\xi] ~dx \quad \forall \ \xi \in \delta^\gamma L^\infty(\Omega).
	\end{equation}
\end{definition}
\begin{remark}
	The integrals in \eqref{eq:notionsol} are well defined. Indeed, the term on the left-hand side of \eqref{eq:notionsol} is well defined since $u \in L^1(\Omega,\delta^\gamma)$ and $\xi \in \delta^\gamma L^\infty(\Omega)$. For the term on the right-hand side, we deduce from \cite[Theorem 3.4]{AbaGomVaz_2019} that
	$|\mathbb{G}^\Omega [\xi]| \leq C\delta^\gamma$ in $\Omega$.
	 This, together with the condition \eqref{eq:notion0}, implies
	\begin{equation*}
	\int_{\Omega} \frac{1}{u^q} \mathbb{G}^\Omega[\xi] ~dx \leq C \int_{\Omega} \frac{1}{u^q} \delta^{\gamma} ~dx < +\infty.
	\end{equation*}
\end{remark}

The value of $q$ plays a crucial role in the analysis of problem \eqref{sing:pure}. An interesting feature of the study of \eqref{sing:pure} lies at the involvement of two critical exponents $q^{\ast}_{s, \gamma}$ and $q^{\ast \ast}_{s, \gamma}$ given by
\[ q^{\ast}_{s,\gamma}:= \frac{2s}{\gamma} -1, \qquad q^{\ast  \ast}_{s,\gamma}:=
\begin{aligned}
	\left\{
	\begin{array}{ll}
		+ \infty &\text{ if } \ \gamma \geq 2s-1,\\
		\displaystyle{\frac{\gamma+1}{2s-\gamma-1}}  &\text{ if } \ \gamma < 2s-1.
	\end{array}
	\right.
\end{aligned}
\]
These exponents classify the weak and strong singular nonlinearities in the class of nonlocal problem \eqref{sing:pure}. This classification is reflected via the behavior of the weak-dual solution near the boundary.

Denote
$$\mathbb{E}:= \left\{(q,\gamma) \in \mathbb{R}^+ \times (0, 2s) : 0 < q < q^{\ast}_{s,\gamma}, \gamma \geq s- \frac{1}{2} \ \text{or} \ q \in ( q^{\ast}_{s,\gamma}, q^{\ast \ast}_{s, \gamma} ), \gamma > s- \frac{1}{2} \right \}.$$
\textit{Throughout the remaining text, we will assume that $(q, \gamma) \in \mathbb{E}$}.

Our first main result is the existence and uniqueness of the weak-dual solution to the purely singular problem \eqref{sing:pure}.
\begin{theorem}\label{thm:pure:sing}
Assume \eqref{L1}-\eqref{L4} and \eqref{G1}-\eqref{G3} hold and $(q, \gamma) \in \mathbb{E}$. Then there exists a unique weak-dual solution $u_*$ of the problem \eqref{sing:pure}.
\end{theorem}

The boundary behavior of the weak-dual solution to problem \eqref{sing:pure} is depicted in terms of the distance function in the following theorem.
\begin{theorem}\label{thm:pure:sing:reg}
Let $u_*$ be the weak-dual solution of the problem \eqref{sing:pure} obtained in Theorem \ref{thm:pure:sing}.

{\sc I. Weakly singular nonlinearity.} If $q \in (0, q^\ast_{s, \gamma})$ then
    \[ u_* \in \mathcal{A}^{-}(\Omega):= \left\{u: u \asymp \delta^\gamma \right\}.\]

{\sc II. Critical singular nonlinearity.} If $q = q^\ast_{s, \gamma}$ then
    \[ u_* \in 	\begin{aligned}
	\mathcal{A}^{0}(\Omega):= \left\{
	\begin{array}{ll}
	\{u: u \asymp \delta^\gamma \ln(\frac{d_\Omega}{\delta})^{\frac{\gamma}{2s}}\} & \text{ if } \ s \leq \gamma , \vspace{0.1cm}\\
	\{u: \delta^\gamma \lesssim u \lesssim \delta^\gamma \ln(\frac{d_{\Omega}}{\delta}) \} & \text{ if } \ s-\frac{1}{2}  <\gamma <s.
	\end{array}
	\right.
	\end{aligned}
	\]

{\sc III. Strongly singular nonlinearity.} If $q \in (q^\ast_{s, \gamma}, q^{\ast \ast}_{s, \gamma})$ then
	\[
	 u_* \in  \mathcal{A}^{+}(\Omega) :=\left\{u: u \asymp \delta^{\frac{2s}{q+1}} \right\}.
	\]
\end{theorem}
As we can see from the above theorem,  when $q=q_{s,\gamma}^*$ and $s-\frac{1}{2}<\gamma<s$, we have $\delta^\gamma \lesssim u \lesssim \delta^\gamma \ln(\frac{d_{\Omega}}{\delta}) \}$. It would be interesting to see if these estimates are sharp or can be improved.

Put
\begin{equation}\label{puresing:bdrybeh}
	\begin{aligned}
		\mathcal{A}_{\ast}(\Omega) := \left\{
		\begin{array}{ll}
			\mathcal{A}^{-}(\Omega) & \text{ if } q  \in \left(0,q^{*}_{s,\gamma}\right),\\
			\mathcal{A}^{0}(\Omega) & \text{ if } q = q^{*}_{s,\gamma},\\
			\mathcal{A}^{+}(\Omega) & \text{ if } q \in (q^{*}_{s,\gamma}, q^{\ast \ast}_{s, \gamma}).
		\end{array}
		\right.
	\end{aligned}
\end{equation}

Denote
\[
\mathfrak{S}_1:= \|u_*\|_{L^\infty(\Omega)} \quad \text{and} \quad \mathfrak{S}_2:= \|\mathbb{G}^\Omega[1]\|_{L^\infty(\Omega)}
\]
where $u_*$ is the unique weak-dual solution of \eqref{sing:pure}.

Next, we consider the following nonlocal semilinear problem involving  a singular nonlinearity and a source term
\begin{equation}\label{sing:f(v)} \tag{$S_{\lambda,q}$}
\left\{
\begin{aligned}
\mathbb{L} v
&=\frac{1}{v^q} + \lambda f(v) \quad &&\text{in } \Omega, \\
v &> 0 &&\text{in } \Omega, \\
v & = 0 &&\text{on } \partial\Omega \text{ or in } \Omega \text{ if applicable},
\end{aligned}
\right.
\end{equation}
where the pair $(\lambda, f)$ satisfies the following assumptions:
\begin{enumerate}[resume,label=(f\arabic{enumi}),ref=f\arabic{enumi}]
	\item \label{f1} $f : (0, \infty) \to \mathbb{R}^+$  is a bounded domain map, {\it i.e.} for any bounded set $A \subset (0, \infty)$, $f(A)$ is bounded;
\end{enumerate}

\begin{enumerate}[resume,label=(f\arabic{enumi}),ref=f\arabic{enumi}]
	\item \label{f2} There exist constants $\Lambda$ and $C_\Lambda >0$ such that
	\[
	\lambda f(t) \leq \Lambda \ \text{for all } \ t \in [0,c_1+ \Lambda c_2] \ \text{and} \ \lambda \in (0, C_\Lambda)\]
	for some $c_1 \geq \mathfrak{S}_1$ and $ c_2 \geq \mathfrak{S}_2$;
\end{enumerate}

\begin{enumerate}[resume,label=(f\arabic{enumi}),ref=f\arabic{enumi}]
	\item \label{f3} the map $t \to \kappa t + \lambda f(t)$ is increasing in $(0, \mathfrak{S}_1 + \Lambda \mathfrak{S}_2]$ for some $\kappa >0.$
\setcounter{enumi}{0}
\end{enumerate}
\medskip

We present below some examples of the source term satisfying conditions \eqref{f1}--\eqref{f3}.

\noindent \textbf{Examples.}
	\begin{enumerate}
		\item[(i)] Let $f(t)= t^r$ for $t > 0$ and $0<r<1$. Then, for any $\lambda >0$, the pair $(\lambda, f)$ satisfies the assumptions \eqref{f1}--\eqref{f3}. Indeed, for $\Lambda>0$, put
		\begin{equation*}
			C_\Lambda := \frac{\Lambda}{(c_1 + c_2 \Lambda)^r}.
		\end{equation*}
Note that $C_{\Lambda} \to +\infty$ as $\Lambda \to +\infty$. One can choose $\Lambda >0$ large enough such that $\lambda f(c_1 + c_2 \Lambda) \leq \Lambda$ for $\lambda \in (0, C_\Lambda).$
		\item[(ii)] Let $f(t)= t^p$ for $t > 0$ and $p \geq 1.$ Then there exists  $\lambda^* >0$ such that for all $\lambda \in (0,\lambda^*)$ the pair $(\lambda, f)$ satisfies the assumption \eqref{f1}--\eqref{f3}. Precisely, we choose $\lambda^*= \frac{1}{c_2}$ if $p=1$ and for $\lambda^*= \frac{(p-1)^{p-1}}{c_1^{p-1} c_2 p^p} $ if $p>1$.
	\end{enumerate}

Weak-dual sub- and supersolutions  of problem \eqref{sing:f(v)} are defined as follows.
\begin{definition}\label{def:sour}
A positive function $v$ is said to be a weak-dual subsolution (resp. weak-dual supersolution) of \eqref{sing:f(v)} if,
	\begin{equation*}\label{def:solf(u)1}
	v \in L_0^{1}(\Omega, \delta^\gamma), \quad  
	\ v^{-q}, f(v) \in L^{1}(\Omega, \delta^\gamma) \\
	\end{equation*}
	 and
	 \begin{equation*}\label{def:solf(u)12}
	\int_{\Omega} v \xi ~dx \leq \ (\text{resp.}\ \geq ) \int_{\Omega} \frac{1}{v^q} \mathbb{G}^\Omega[\xi] ~dx + \lambda \int_{\Omega} f(v) \mathbb{G}^\Omega[\xi] ~dx, \quad \forall \ \xi \in \delta^\gamma L^\infty(\Omega), \ \xi \geq 0.
	\end{equation*}
	\noindent
	A function which is both weak-dual sub-  and supersolution of \eqref{sing:f(v)} is called a weak-dual solution to \eqref{sing:f(v)}.
\end{definition}
\noindent
\begin{theorem}\label{thm:source}
Assume \eqref{L1}--\eqref{L4}, \eqref{G1}--\eqref{G3} and \eqref{f1}--\eqref{f3} hold, and $(q, \gamma) \in \mathbb{E}$. Then there exists a weak-dual solution $v$ of \eqref{sing:f(v)}. Moreover, $v \in \mathcal{A}_{\ast}(\Omega).$
\end{theorem}

Next we are concerned with the following nonlocal semilinear problem involving  a singular nonlinearity and an absorption term
\begin{equation}\label{sing+absorption}
\tag{$A_{g, q}$}
\left\{
\begin{aligned}
\mathbb{L} w + g(w)
&=\frac{1}{w^q}  \quad &&\text{in } \Omega, \\
w &> 0 &&\text{in } \Omega, \\
w & = 0 &&\text{on } \partial\Omega \text{ or in } \Omega \text{ if applicable},
\end{aligned}
\right.
\end{equation}
where the function $g$ satisfy the following conditions:
\begin{enumerate}[resume,label=(g\arabic{enumi}),ref=g\arabic{enumi}]
	\item \label{g1}  $g : (0, \infty) \to \mathbb{R}^+$  is a bounded domain map;
\end{enumerate}

\begin{enumerate}[resume,label=(g\arabic{enumi}),ref=g\arabic{enumi}]
	\item \label{g2} the map $t \to \mu t - g(t)$ is increasing in the interval $(0, \mathfrak{S}_1]$ for some $\mu >0$;
\end{enumerate}

\begin{enumerate}[resume,label=(g\arabic{enumi}),ref=g\arabic{enumi}]
	\item \label{g3}
     $\lim_{t \to 0^+} t^p g(t) < +\infty$ with $p<q$.
\end{enumerate}

\noindent \textbf{Examples.}
	\begin{enumerate}
		\item[(i)] Let $g(t)= t^p$ for $t > 0$ and $p \geq 1.$ Then, for any $\mu > p \mathfrak{S}_1^{p-1}$, the function $g$ satisfies the assumptions \eqref{g1}--\eqref{g3}.
		\item[(ii)] Let $g(t)= t^{-r}$ for $t > 0$ and $0< r < q$. Then $g$ satisfies the assumptions \eqref{g1}--\eqref{g3}.
	\end{enumerate}
The definition of weak-dual sub- and supersolutions is given below.
\begin{definition}\label{def:abs}
	A positive function $w$ is said to be a weak-dual subsolution (resp.  weak-dual supersolution) of \eqref{sing+absorption} if
	\begin{equation*}\label{def:solg(u)1}
	w \in L_0^{1}(\Omega, \delta^\gamma), \quad  
	w^{-q}, g(w) \in L^{1}(\Omega, \delta^\gamma) \\
	\end{equation*}
	and
	\begin{equation*}\label{def:solg(u)2}
	\int_{\Omega} w \xi ~dx + \int_{\Omega} g(w) \mathbb{G}^\Omega[\xi] ~dx \leq \ (\text{resp.}\ \geq ) \int_{\Omega} \frac{1}{w^q} \mathbb{G}^\Omega[\xi] ~dx,  \quad \forall \ \xi \in \delta^\gamma L^\infty(\Omega), \ \xi \geq 0.
	\end{equation*}
	\noindent
	A function which is both weak-dual sub-  and super weak-dual solution of \eqref{sing+absorption} is called a weak-dual solution to \eqref{sing+absorption}.
\end{definition}
\begin{theorem}\label{thm:absor}
Assume \eqref{L1}--\eqref{L4}, \eqref{G1}--\eqref{G3}, \eqref{g1}--\eqref{g3} hold, and  $(q, \gamma) \in \mathbb{E}$. Then there exists a unique weak-dual solution $w$ of \eqref{sing+absorption}. Moreover, $w \in \mathcal{A}_{\ast}(\Omega)$.
\end{theorem}

We end this Section with an open question. We note that the case $q \geq q_{s,\gamma}^{\ast \ast}$ and $\gamma < 2s-1$ in not included in this paper due to the absence of boundary estimates of the solution to approximating problems. Therefore, it would be interesting to investigate this case.

\section{Preliminary results and Kato type inequalities}\label{prelim-kato-ineq}
\subsection{Preliminary results}
In this subsection, we recall some basic results comprised of lower Hopf type estimate, integration by parts formula and action of the Green kernel on power of the distance functions.
\begin{lemma}[Theorem 2.6, \cite{AbaGomVaz_2019}]\label{Hopf}
	Assume \eqref{G1} and \eqref{G2} hold. Then there exists $c>0$ such that for any $0 \leq f \in L^1(\Omega,\delta^\gamma)$,
	$$\mathbb{G}^\Omega[f](x) \geq c \delta(x)^{\gamma} \int_\Omega f(y) \delta(y)^\gamma ~dy, \ x \in \Omega.$$
\end{lemma}
\begin{lemma}[Lemma 4.5, \cite{ChaGomVaz_2020}]\label{IBP}
	Assume \eqref{G1} and\eqref{G2} hold, $f \in L^1(\Omega, \delta^\gamma)$ and $\xi \in \delta^\gamma L^\infty(\Omega)$. Then we have
	$$\int_{\Omega} \mathbb{G}^\Omega[f] \xi ~dx= \int_{\Omega} f \mathbb{G}^\Omega[\xi]  ~dx.$$
\end{lemma}
\begin{lemma}[Theorem 3.4, \cite{AbaGomVaz_2019}]\label{est:green}
	Assume Assume \eqref{G1} and \eqref{G2} hold and $\beta < \gamma +1$. Then $\delta^{-\beta} \in L^1(\Omega, \delta^\gamma)$ and
	\begin{equation}\label{bdry:beha}
	\begin{aligned}
	\mathbb{G}^\Omega[\delta^{-\beta}] \asymp  \left\{
	\begin{array}{ll}
	\delta^\gamma &  \text{ if } \gamma < 2s-\beta, \\
	\delta^\gamma \ln(\frac{d_{\Omega}}{\delta}) \,  & \text{ if } \gamma = 2s-\beta  \text{ and } 2 \gamma > 2s -1,\\
	\delta^{2s-\beta} & \text{ if } \gamma > 2s-\beta  \text{ and } 2 \gamma > 2s -1.\\
	\end{array}
	\right.
	\end{aligned}
	\end{equation}
	Here $d_{\Omega}=2\diam(\Omega)$.
\end{lemma}
\begin{lemma}[Lemma 5.2, \cite{HuNg_absorption}]\label{lip-prop}
    Assume $u, v \in H_{00}^s(\Omega)$ and $h : \mathbb{R} \to \mathbb{R}$ is a Lipschitz function such that $h(0) =0.$ Then $h(u) \in H_{00}^s(\Omega)$ and $ uv \in H_{00}^s(\Omega)$ if $u, v \in L^\infty(\Omega).$
\end{lemma}
\subsection{Kato type inequalities}
In this subsection, we prove a Kato type inequality expressed in terms of the Green operator under assumptions \eqref{L1}--\eqref{L4} and \eqref{G1}--\eqref{G2}, which plays the role of the comparison principle in our study of weak-dual solutions. The proof of the inequality is based on the convexity inequality \eqref{L4} and ideas recently developed in \cite[Lemma 5.2 and Theorem 3.2]{HuNg_absorption}. It is worth stressing that the following Kato type inequality seems to have a wider applicability in comparison with the inequality obtained in \cite[Theorem 3.2]{HuNg_absorption}.
\begin{proposition}
\label{Kato}
Assume \eqref{L1}--\eqref{L4} and \eqref{G1}--\eqref{G2} hold, $f \in L^1(\Omega, \delta^\gamma)$ and $u= \mathbb{G}^\Omega[f]$. Then
\begin{equation} \label{kato||}
	\int_\Omega |u| \xi ~dx \leq \int_{\Omega} \sign(u) \mathbb{G}^\Omega[\xi] f ~dx,
\end{equation}
and
\begin{equation} \label{kato+}
	\int_\Omega u^+ \xi ~dx \leq \int_{\Omega} \sign^+(u) \mathbb{G}^\Omega[\xi] f ~dx,
\end{equation}
for every $\xi \in  \delta^\gamma L^\infty(\Omega)$ such that $\mathbb{G}^\Omega[\xi] \geq 0$ a.e. in $\Omega.$
\end{proposition}
\begin{proof} The proof is based on two claims.
    \vspace{0.1cm}\\
    \textbf{Claim 1:} Assume $f \in C_c^\infty(\Omega)$, $u=\mathbb{G}^\Omega[f]$ and $p \in C^{1,1}(\mathbb{R})$ is a convex function  such that $p(0)=p'(0) =0$ and $|p'| \leq 1$. Then, for any $\xi \in \delta^\gamma L^\infty(\Omega), \mathbb{G}^\Omega[\xi] \geq 0$, we have
    \begin{equation}\label{kato-est1}
         \int_{\Omega} p(u) \xi ~dx \leq \int_{\Omega} f p'(u) \mathbb{G}^\Omega [\xi] ~dx.
    \end{equation}

In order to prove \eqref{kato-est1}, we  will employ the following equality, derived from \cite[Proposition 5.2]{HuNg_source}, which asserts  that for any $g \in L^2(\Omega)$,  $\mathbb{G}^\Omega[g] \in \mathbb{H}(\Omega)$ and
\begin{equation}\label{equi:notion}
    \int_{\Omega} g \zeta ~dx = \langle \mathbb{G}^\Omega[g] , \zeta \rangle_{\mathbb{H}(\Omega)}, \quad \forall\, \zeta \in \mathbb{H}(\Omega).
\end{equation}

Take $\xi \in \delta^\gamma L^\infty(\Omega)$ with $\mathbb{G}^\Omega[\xi] \geq 0$. Then it can be easily checked that $\xi \in L^2(\Omega)$ and hence  $u=\mathbb{G}^\Omega[f] \in \mathbb{H}$. Consequently, by Lemma \ref{lip-prop}, we have $p(u) \in \mathbb{H}$. Therefore, replacing $g$ by $\xi$ and $\zeta$ by $p(u)$ in \eqref{equi:notion}, we deduce that
\begin{equation} \label{KT-1}
\int_{\Omega} p(u) \xi ~dx = \left\langle p(u), \mathbb{G}^{\Omega}[\xi]  \right\rangle_{\mathbb{H}(\Omega)}.
\end{equation}

Next, by using \cite[Propositions 4.11 and 5.2]{HuNg_source}, assumption \eqref{G2} and Lemma \ref{lip-prop}, we have $u = \mathbb{G}^\Omega[f] \in \mathbb{H}(\Omega)$, $p'(u) \in \mathbb{H}(\Omega) \cap L^\infty(\Omega),$ and $p'(u) \mathbb{G}^\Omega[\xi] \in \mathbb{H}(\Omega)$. Now, by taking $g=f$ and $\zeta = p'(u) \mathbb{G}^\Omega[\xi]$ in \eqref{equi:notion}, and using \eqref{L4}, we obtain
 \begin{equation} \label{KT-2}
    \int_{\Omega} f p'(u) \mathbb{G}^\Omega[\xi] ~dx =
     \left\langle  \mathbb{G}^{\Omega}[f],  p'(u)  \mathbb{G}^{\Omega}[\xi] \right\rangle_{\mathbb{H}(\Omega)} = \left\langle  u,  p'(u)  \mathbb{G}^{\Omega}[\xi] \right\rangle_{\mathbb{H}(\Omega)}
   \geq  \left\langle p(u), \mathbb{G}^{\Omega}[\xi]  \right\rangle_{\mathbb{H}(\Omega)}.
 \end{equation}
Gathering \eqref{KT-1} and \eqref{KT-2}, we deduce that
\[
\int_{\Omega} f p'(u) \mathbb{G}^\Omega[\xi] ~dx
   \geq  \left\langle p(u), \mathbb{G}^{\Omega}[\xi]  \right\rangle_{\mathbb{H}(\Omega)} = \int_{\Omega} p(u) \xi ~dx,
\]
which yields \eqref{kato-est1}.
\vspace{0.1cm}\\
\textbf{Claim 2:} Claim 1 still holds true if  $f \in L^1(\Omega, \delta^\gamma)$.

Indeed, assume $f \in L^1(\Omega, \delta^\gamma)$ and take $\xi \in \delta^\gamma L^\infty(\Omega)$ with $\mathbb{G}^\Omega[\xi] \geq 0$. Let $\{f_n\}_{n \in \mathbb{N}} \subset C_c^\infty(\Omega)$ be a sequence of functions such that
$f_n \to f$ in $L^1(\Omega, \delta^\gamma)$ and a.e. in $\Omega$.

Put $u_n:= \mathbb{G}^\Omega[f_n]$. Then by Claim 1, we have
\begin{equation}\label{kato-est1-un}
	\int_{\Omega} p(u_n) \xi ~dx \leq \int_{\Omega} f p'(u_n) \mathbb{G}^\Omega [\xi] ~dx.
\end{equation}

Recall that $u= \mathbb{G}^\Omega[f]$. Since the map $\mathbb{G}^\Omega:  L^1(\Omega, \delta^\gamma) \to  L^1(\Omega, \delta^\gamma)$ is continuous and $p \in C^{1,1}(\mathbb{R})$, $|p'| \leq 1$, we deduce that, up to a subsequence, $u_n \to u$ in $ L^1(\Omega, \delta^\gamma)$ and a.e. in $\Omega$, and $p(u_n) \to p(u)$ in $ L^1(\Omega, \delta^\gamma)$ and a.e. in $\Omega$. It follows that
\begin{equation}\label{kato-est2}
 \lim_{n \to \infty} \int_{\Omega} p(u_n) \xi~dx = \int_{\Omega} p(u) \xi~dx.
\end{equation}
Now, again by using $|p'(u_n)| \leq 1$, $f_n p'(u_n) \to fp'(u)$ a.e. in $\Omega$ and the generalized Lebesgue dominated convergence theorem, we obtain $f_n p'(u_n) \to f p'(u)$ in $L^1(\Omega, \delta^\gamma).$ Moreover, since the map $\mathbb{G}^\Omega : \delta^\gamma L^\infty(\Omega) \to  \delta^\gamma L^\infty(\Omega)$ is continuous (see \cite[Proposition 3.5]{ChaGomVaz_2020}), we have $\mathbb{G}^\Omega[\xi]$ in $ \delta^\gamma L^\infty(\Omega)$. Therefore
\begin{equation}\label{kato-est3}
\lim_{n \to \infty}    \int_{\Omega} f_n p'(u_n) \mathbb{G}^\Omega[\xi] ~dx = \int_{\Omega} f p'(u) \mathbb{G}^\Omega[\xi] ~dx.
\end{equation}
By letting $n \to \infty$ in \eqref{kato-est1-un} and using \eqref{kato-est2} and \eqref{kato-est3}, we get the required claim.

Next we will prove inequality \eqref{kato||}. Consider the sequence $\{ p_k \}_{k \in \mathbb{N}}$ given by
\begin{equation}\label{eq:sequencepk}
	p_k(t) := \left\{ \begin{aligned} &|t| - \frac{1}{2k} &&\text{ if }|t| \ge \frac{1}{k},\\[3pt]
		&\frac{kt^2}{2} &&\text{ if } |t| < \frac{1}{k}.
	\end{aligned} \right.
\end{equation}
Then for every $k \in \mathbb{N}$, $p_k \in C^{1,1}(\mathbb{R})$ is convex, $p_k(0) = (p_k)'(0) = 0$ and $|(p_k)'| \le 1$.  Hence, employing Claim 2  with $p = p_k$, one has
\begin{equation} \label{pku}
	\int_{\Omega} p_k(u) \xi ~dx \le \int_{\Omega} f (p_k)'(u) \mathbb{G}^{\Omega}[\xi] ~dx,\quad \forall \xi \in \delta^{\gamma}L^{\infty}(\Omega), \mathbb{G}^{\Omega} [\xi]\ge 0.
\end{equation}
Notice that $p_k(t) \to |t|$ and $(p_k)'(t) \to \sign (t)$ as $k \to \infty$.  Hence, letting $k \to \infty$ in \eqref{pku} and using the dominated convergence theorem, we obtain \eqref{kato||}.

Finally, inequality \eqref{kato+} follows from inequality \eqref{kato||} and the integration by parts formula (see Lemma \ref{IBP}). The proof is complete.
\end{proof}
\begin{proposition}
Assume \eqref{G1}-\eqref{G2} hold and $\beta < \gamma +1$. Then there exists a unique weak-dual solution $\phi_{\beta} \in L^1_0(\Omega, \delta^\gamma)$ of the following problem
\begin{equation} \label{sing:weight}
	\left\{
	\begin{aligned}
	\mathbb{L} \phi
	& =\frac{1}{\delta^{\beta}} \quad &&\text{in } \Omega, \\
	\phi &>0 \quad &&\text{in } \Omega, \\
	\phi & {}= 0
	&& \text{ on } \partial\Omega \text{ or in } \Omega^c  \text{ if applicable},
	\end{aligned}
	\right.
	\end{equation}
	in the sense that
	$$\int_{\Omega} \phi_{\beta} \xi ~dx = \int_{\Omega} \frac{1}{\delta^\beta} \mathbb{G}^\Omega[\xi]~dx, \quad \forall \, \xi \in \delta^\gamma L^\infty(\Omega).$$
	Moreover, $\phi_{\beta}$ admits the behavior as  in \eqref{bdry:beha}.
\end{proposition}
\begin{proof}
	Define
	\begin{equation}\label{weight:solution}
	\begin{aligned}
	\phi_{\beta}(x) :=  \left\{
	\begin{array}{ll}
	\mathbb{G}^\Omega \displaystyle{\left[\frac{1}{\delta^\beta}\right]} &  \text{ if } x \in \Omega, \\
	0
	& \text{ if } x \in \partial\Omega\, \ \mbox{or $\Omega^c$ if applicable}.\\
	\end{array}
	\right.
	\end{aligned}
	\end{equation}
	Then by using Lemma \ref{IBP} and Lemma \ref{est:green}, we obtain $\phi_{\beta} \in L_0^1(\Omega, \delta^
	\gamma)$, $\phi_{\beta}$ admits the behavior as  in \eqref{bdry:beha} and
	$$\int_{\Omega} \phi_{\beta} \xi ~dx = \int_{\Omega} \mathbb{G}^\Omega \left[\frac{1}{\delta^\beta}\right] \xi ~dx= \int_{\Omega} \frac{1}{\delta^\beta} \mathbb{G}^\Omega[\xi]~dx, \quad \forall \, \xi \in  \delta^\gamma L^\infty(\Omega).$$
Moreover, $\phi_{\beta}$ is the unique weak-dual solution of \eqref{sing:weight} due to \cite[Theorem 2.5]{AbaGomVaz_2019} and an approximation argument.
\end{proof}
\section{Purely singular problem}\label{PS-problem}
In this section, we focus on the  purely singular problem \eqref{sing:pure}. To this purpose, first we prove some estimates of Green operator acting over singular perturbed terms and then by using the approximation method and Kato type inequality, we establish the existence, uniqueness and boundary behavior of the weak-dual solution to \eqref{sing:pure}.
\subsection{Estimates on Green kernel}
Let $s \in (0,1]$ and $q>0$. Denote
\begin{equation} \label{alphabeta}\beta:= \frac{2sq}{q+1},  \qquad \alpha:=2s-\beta = \frac{2s}{q+1},
\end{equation}
and for $\eta>0$, put
\begin{equation} \label{Oeta} \Omega_{\eta}:= \{x \in \Omega: \delta(x) < \eta\}.
\end{equation}
\begin{lemma}\label{lem:greenest}
Assume \eqref{G1} and \eqref{G2} hold. Then, for $\epsilon, \eta \in (0,1)$ and $\gamma > \alpha$, there exist positive constants $c_1= c_1(\beta)$ and $c_2=c_2(\diam(\Omega), \eta, \gamma, \beta)$ such that the following estimates hold:
	\begin{equation}\label{est:green:lower:1}
	\mathbb{G}^{\Omega}\left[\frac{1}{(\delta+\epsilon^{\frac{1}{\alpha}})^\beta} {\bf 1}_{\Omega_{\eta}}\right](x) \geq c_1 \left(\frac{1}{2} (\delta(x)+ \epsilon^{\frac{1}{\alpha}})^\alpha - \epsilon\right), \quad  \forall \, x \in \Omega_{\frac{\eta}{2}},
	\end{equation}
	and
	\begin{equation}\label{est:green:lower:2}
	\mathbb{G}^{\Omega}\left[\frac{1}{(\delta+\epsilon^{\frac{1}{\alpha}})^\beta} \right](x) \geq c_2 (\delta(x)+ \epsilon^{\frac{1}{\alpha}})^\alpha - \epsilon, \quad  \forall \, x \in \Omega_{\frac{\eta}{2}}^c .
	\end{equation}
Here ${\bf 1}_{\Omega_\eta}$ denotes the characteristic function of $\Omega_\eta$.
\end{lemma}
\begin{proof}
	
First we will prove estimate \eqref{est:green:lower:1}.
	Denote $\epsilon_1= \epsilon^{\frac{1}{\alpha}}$. By assumption \eqref{G2}, we have, for $x \in \Omega_{\frac{\eta}{2}}$
	\begin{equation}\label{est:lem:1}
	\begin{split}
	\mathbb{G}^{\Omega}&\left[\frac{1}{(\delta+\epsilon_1)^\beta} {\bf 1}_{\Omega_{\eta}}\right](x) = \int_{\Omega_{\eta}} \frac{G^{\Omega}(x,y)}{(\delta(y)+ \epsilon_1)^\beta} ~dy\\
	& \geq C \int_{\Omega \cap B(x,\frac{\delta(x)}{2})} \frac{1}{(\delta(y)+ \epsilon_1)^\beta} \frac{1}{|x-y|^{N-2s}}  \left( \frac{\delta(x)}{|x-y|} \wedge 1 \right)^{\gamma} \left( \frac{\delta(y)}{|x-y|} \wedge 1 \right)^{\gamma} \,dy.
	\end{split}
	\end{equation}
Here we note that $\Omega \cap B(x,\frac{\delta(x)}{2}) \subset \Omega_{\eta}$. Now for $x \in  \Omega_\frac{\eta}{2}$ and $y \in \Omega \cap B(x,\frac{\delta(x)}{2})$, we have
	\begin{equation}\label{est:lem:lower}
	\left( \frac{\delta(x)}{|x-y|} \wedge 1 \right)^{\gamma} \geq 1, \left( \frac{\delta(y)}{|x-y|} \wedge 1 \right)^{\gamma} \geq 1 \quad \text{and }\ \frac{1}{(\delta(y)+ \epsilon_1)^{\beta}} \geq \left(\frac{2}{3}\right)^\beta \frac{1}{(\delta(x)+ \epsilon_1)^{\beta}}.
	\end{equation}
	Combining \eqref{est:lem:1} and \eqref{est:lem:lower}, we obtain
	\begin{equation}\label{est:lem:2}
	\mathbb{G}^{\Omega}\left[\frac{1}{(\delta+\epsilon_1)^\beta} {\bf 1}_{\Omega_{\eta}}\right](x)
	\geq \left(\frac{2}{3}\right)^\beta  \frac{C}{(\delta(x)+ \epsilon_1)^\beta} \int_{\Omega \cap B(x,\frac{\delta(x)}{2})}  \frac{1}{|x-y|^{N-2s}} ~dy = \frac{C \delta(x)^{2s}}{(\delta(x)+ \epsilon_1)^\beta} .
	\end{equation}
	By applying the inequality
	$$(a+b)^{2s} \leq \max\{1,2^{2s-1}\}(a^{2s} + b^{2s}), \quad a \geq 0, b\geq 0,$$
	with $a= \delta(x)$ and $b= \epsilon_1$, we get
	\begin{equation}\label{est:lem:3}
    \delta(x)^{2s} \geq \frac{1}{\max\{1,2^{2s-1}\} }(\delta(x)+\epsilon_1)^{2s} - \epsilon_1^{2s}.
	\end{equation}
	Finally, by using \eqref{est:lem:3} in \eqref{est:lem:2}, we obtain
	\begin{equation*}
	\begin{split}
	\mathbb{G}^{\Omega}\left[\frac{1}{(\delta+\epsilon^{\frac{1}{\alpha}})^\beta} {\bf 1}_{\Omega_{\eta}}\right](x) &=  \mathbb{G}^{\Omega}\left[\frac{1}{(\delta+\epsilon_1)^\beta} {\bf 1}_{\Omega_{\eta}}\right](x)
	 \geq c_1 \left( \frac{1}{2}(\delta(x)+\epsilon_1)^{2s-\beta} -\epsilon_1^{2s-\beta} \right)\\
	&=  c_1 \left( \frac{1}{2}  (\delta(x)+\epsilon^{\frac{1}{\alpha}})^{\alpha} -\epsilon\right) .
	\end{split}
	\end{equation*}

Next we prove estimate \eqref{est:green:lower:2}. Let $x \in \Omega_{\frac{\eta}{2}}^c$. Then by using Lemma \ref{est:green} and assumption $\gamma > \alpha$, we get
	\begin{equation}
	\begin{split}
	\mathbb{G}^{\Omega}\left[\frac{1}{(\delta+\epsilon_1)^\beta} \right](x) & \geq \frac{1}{(d_{\Omega}+ 1)^\beta}  \mathbb{G}^{\Omega}[{\bf 1}_{\Omega}](x) \geq c\delta(x)^{\min\{ \gamma, 2s\} }   \geq c_2 (\delta(x)+ \epsilon^{\frac{1}{\alpha}})^\alpha - \epsilon,
	\end{split}
	\end{equation}
	where $c_2$ depends upon $\eta, \alpha, \gamma$ and $d_{\Omega}$.
\end{proof}
\begin{lemma}\label{lem:greest1}
Assume \eqref{G1} and \eqref{G2} hold and $s-\frac{1}{2} < \gamma < 2s$. Then for $\sigma \in [0,1)$, there holds
	$$\mathbb{G}^\Omega\left[ \frac{1}{\delta^{2s-\gamma}} \ln^{-\sigma}\left(\frac{d_{\Omega}}{\delta}\right) \right](x) \asymp \delta(x)^{\gamma} \ln^{1-\sigma}\left(\frac{d_{\Omega}}{\delta(x)}\right), \quad \forall \, x \in \Omega_{\frac{\eta}{2}},$$
	for some $\eta>0$.
\end{lemma}
\begin{proof} Let $\eta>0$ small. To prove upper and lower estimates, first we split the integrals over two regions $\Omega_\eta$ and $\Omega \setminus \Omega_\eta$ as follows
	\begin{equation} \label{eq:GI1I2}\begin{aligned} \mathbb{G}^\Omega\left[ \frac{1}{\delta^{2s-\gamma}} \ln^{-\sigma}\left(\frac{d_{\Omega}}{\delta}\right) \right](x) &= \int_{\Omega_{\eta}} \frac{G^{\Omega}(x,y)}{\delta(y)^{2s-\gamma}} \ln^{-\sigma}\left(\frac{d_{\Omega}}{\delta(y)}\right)dy \\
	& \quad + \int_{\Omega \setminus \Omega_{\eta}} \frac{G^{\Omega}(x,y)}{\delta(y)^{2s-\gamma}} \ln^{-\sigma}\left(\frac{d_{\Omega}}{\delta(y)}\right)dy \\
	&=: I_1(x) + I_2(x).
	\end{aligned} \end{equation}
	Take $x \in \Omega_\frac{\eta}{2}$ and let $\Phi: B(x,1) \to B(0,1)$ be a diffeomorphism such that
\begin{equation} \label{diffeo}	\begin{aligned}\Phi(\Omega \cap B(x,1)) = B(0,1) \cap \{y \in \mathbb{R}^N : y \cdot e_N >0\}, \\
	\Phi(y) \cdot e_N = \delta(y) \ \text{for} \ y \in B(x,1)\  \text{and} \ \Phi(x)= \delta(x) e_N.
\end{aligned}
\end{equation}

For the first integral $I_1(x)$, we partition the set $\Omega_\eta$ into the following five components
	$$\mathscr{O}_1:= B(x, \delta(x)/2), \quad \mathscr{O}_2:= \Omega_{\eta} \setminus B(x,1),$$
	$$\mathscr{O}_3:= \{y: \delta(y) < \delta(x)/2\} \cap B(x,1), \quad \mathscr{O}_4:= \left\{y: \frac{3 \delta(x)}{2} < \delta(y) < \eta \right\} \cap B(x,1),$$
	$$\mathscr{O}_5:= \left\{y: \frac{\delta(x)}{2} < \delta(y) < \frac{3 \delta(x)}{2}\right\} \cap \left( B(x,1) \setminus B(x, \frac{\delta(x)}{2})\right).$$
	\medskip
	
	\textbf{Lower estimate.} For $y \in \mathscr{O}_4$, we have
	$$\left( \frac{\delta(x) \delta(y)}{|x-y|^{2}} \wedge 1 \right) \asymp \frac{\delta(x) \delta(y)}{|x-y|^{2}}\quad \text{and} \quad \ln^{-\sigma}\left(\frac{d_{\Omega}}{\delta(x)}\right) \leq \ln^{-\sigma}\left(\frac{d_{\Omega}}{\delta(y)}\right).$$
	Therefore
	\begin{equation} \label{est:I1-1} \begin{aligned}
	I_1(x) \geq \int_{\mathscr{O}_4}  \frac{G^{\Omega}(x,y)}{\delta(y)^{2s-\gamma}} \ln^{-\sigma}\left(\frac{d_{\Omega}}{\delta(y)}\right)dy
	\geq  \ln^{-\sigma}\left(\frac{d_{\Omega}}{\delta(x)}\right) \int_{\mathscr{O}_4} \frac{G^\Omega(x,y)}{\delta(y)^{2s-\gamma}} ~dy
	\end{aligned} \end{equation}
	Now, by performing a change of variables via diffeomorphism $\Phi$ in \eqref{diffeo} and using estimates on \cite[Proof of Lemma 3.3, Page 40]{AbaGomVaz_2019}, we get
	\begin{equation*}
	\begin{split}
	 \int_{\mathscr{O}_4} \frac{G^\Omega(x,y)}{\delta(y)^{2s-\gamma}} ~dy  &\gtrsim  \delta(x)^{\gamma} \int_{\{\frac{3\delta(x)}{2} < z_N < \eta\} \cap B(0,1)} \frac{z_N^{2\gamma -2s}}{(|\delta(x)-z_N| + |z'|)^{N-2s+2\gamma}} ~dz_N dz'\\
	& = \delta(x)^{\gamma} \int_{3/2}^{\eta/\delta(x)}  \int_0^{1/\delta(x)} \frac{t^{N-2} h^{2\gamma-2s}}{(|1-h|+t)^{N-2s+2\gamma}} ~dt~dh\\
	& =  \delta(x)^{\gamma} \int_{3/2}^{\eta/\delta(x)} \frac{h^{2\gamma-2s}}{(h-1)^{1-2s+2\gamma}} \int_0^{1/(h-1)\delta(x)} \frac{ r^{N-2} }{(1+r)^{N-2s+2\gamma} } ~dr~dh \\
	& \gtrsim \delta(x)^{\gamma} \int_{3/2}^{\eta/\delta(x)} \frac{h^{2\gamma-2s}}{(h-1)^{1-2s+2\gamma}} \int_1^{1/(h-1)\delta(x)} \frac{1 }{(1+r)^{2-2s+2\gamma} } ~dr~dh \\
	&\gtrsim \delta(x)^{\gamma} \int_{3/2}^{\eta/\delta(x)} \frac{h^{2\gamma-2s}}{(h-1)^{1-2s+2\gamma}} ~dh \\
	&\gtrsim  \delta(x)^{\gamma} \int_{3/2}^{\eta/\delta(x)}  \frac{1}{h} ~dh  = \delta(x)^\gamma \left(\ln\left(\frac{\eta}{\delta(x)}\right) - \ln\left(\frac{3}{2} \right) \right).
	\end{split}
	\end{equation*}
 This implies that there exists a constant $C$ independent of the parameter $\sigma$
	such that
	\begin{equation} \label{est:I1-2} \int_{\mathscr{O}_4} \frac{G^\Omega(x,y)}{\delta(y)^{2s-\gamma}} ~dy \geq C \delta(x)^\gamma \ln\left(\frac{d_{\Omega}}{\delta(x)}\right).
	\end{equation}
	By combining \eqref{eq:GI1I2}, \eqref{est:I1-1} and \eqref{est:I1-2}, we obtain
	$$ \mathbb{G}^\Omega\left[ \frac{1}{\delta^{2s-\gamma}} \ln^{-\sigma}\left(\frac{d_{\Omega}}{\delta}\right) \right](x) \geq I_1(x) \geq C \delta(x)^\gamma \ln^{1-\sigma}\left(\frac{d_{\Omega}}{\delta(x)}\right).$$
	Thus, we obtain the lower bound.
	\medskip
	
	\textbf{Upper estimate.} By partitioning the domain of integral $I_1(x)$ over $\{\mathscr{O}_i\}_{i=1}^5$, we find the upper estimates over each subdomain $\mathscr{O}_i$.  Observing, for $y_1 \in \mathscr{O}_1$ and $y_2 \in \cup_{i=2}^5 \mathscr{O}_i$, we have
	$$\left( \frac{\delta(x) \delta(y_1)}{|x-y_1|^{2}} \wedge 1 \right) \asymp 1\quad \text{and} \quad \left( \frac{\delta(x) \delta(y_2)}{|x-y_2|^{2}} \wedge 1 \right) \asymp \frac{\delta(x) \delta(y_2)}{|x-y_2|^{2}}.$$
	Now, again by using the change of variables via diffeomorphism $\phi$ and \cite[Proof of Lemma 3.3]{AbaGomVaz_2019}, we get estimates in each domain. \medskip
	
	\textit{Upper bound in $\mathscr{O}_1$.} Choosing $\eta$ small enough such that $0< \eta < \frac{2d_{\Omega}}{\exp(1)}$, for some  $c \in (0,1)$, we have, for any $y \in \mathscr{O}_1$,
	$$\frac{1}{2} \delta(x) \leq \delta(y) \leq \frac{3}{2} \delta(x), \quad \ln^{-\sigma}\left(\frac{d_{\Omega}}{\delta(y)}\right) \leq c^{-\sigma} \ln^{-\sigma}\left(\frac{d_{\Omega}}{\delta(x)}\right) \leq c^{-1} \ln^{-\sigma}\left(\frac{d_{\Omega}}{\delta(x)}\right). $$
	Therefore
	\begin{equation*}
	\begin{split}
	\int_{\mathscr{O}_1}  \frac{G^{\Omega}(x,y)}{\delta(y)^{2s-\gamma}} \ln^{-\sigma}\left(\frac{d_{\Omega}}{\delta(y)}\right)dy &\leq c^{-1} \ln^{-\sigma}\left(\frac{d_{\Omega}}{\delta(x)}\right) \delta(x)^{\gamma-2s} \int_{B(x, \delta(x)/2)} \frac{1}{|x-y|^{N-2s}} ~dy \\
	& \leq c^{-1} \ln^{-\sigma}\left(\frac{d_{\Omega}}{\delta(x)}\right) \delta(x)^{\gamma} \\
	& \leq C \ln^{1-\sigma}\left(\frac{d_{\Omega}}{\delta(x)}\right) \delta(x)^{\gamma},
	\end{split}
	\end{equation*}
	where $C$ is independent of parameter $\sigma$.  \medskip
	
	\textit{Upper bound in $\mathscr{O}_2$.} We have
	\begin{equation*}
	\begin{split}
	\int_{\mathscr{O}_2} \frac{G^{\Omega}(x,y)}{\delta(y)^{2s-\gamma}} \ln^{-\sigma}\left(\frac{d_{\Omega}}{\delta(y)}\right)dy &\leq \ln^{-\sigma}\left(\frac{d_{\Omega}}{\eta}\right) \delta(x)^{\gamma} \int_{\mathscr{O}_2} \frac{\delta(y)^{2\gamma-2s}}{|x-y|^{N-2s+2\gamma}}  ~dy\\
	&\leq \ln^{-\sigma}\left(\frac{d_{\Omega}}{\eta}\right) \delta(x)^{\gamma} \int_{\mathscr{O}_2} \delta(y)^{2\gamma-2s} ~dy \\
	& \leq \eta^{2\gamma - 2s+1} \ln^{-\sigma}\left(\frac{d_{\Omega}}{\eta}\right) \delta(x)^{\gamma} \\
	& \leq C(\eta, s, \gamma) \  \ln^{1-\sigma}\left(\frac{d_{\Omega}}{\delta(x)}\right) \delta(x)^{\gamma},
	\end{split}
	\end{equation*}
where in the second last inequality we used the fact that $\gamma > s-\frac{1}{2}$ and $C$ is independent of parameter $\sigma$. \medskip
	
	\textit{Upper bound in $\mathscr{O}_3$.} We note that, for any $y \in \mathscr{O}_3$,
	$$\ln^{-\sigma}\left(\frac{d_{\Omega}}{\delta(y)}\right) \leq \ln^{-\sigma}\left(\frac{d_{\Omega}}{\delta(x)}\right).$$
	Therefore
	\begin{equation*}
	\begin{split}
	\int_{\mathscr{O}_3}  & \frac{G^{\Omega}(x,y)}{\delta(y)^{2s-\gamma}} \ln^{-\sigma}\left(\frac{d_{\Omega}}{\delta(y)}\right)dy \leq \ln^{-\sigma}\left(\frac{d_{\Omega}}{\delta(x)}\right) \delta(x)^{\gamma} \int_{\mathscr{O}_3} \frac{\delta(y)^{2\gamma-2s}}{|x-y|^{N-2s+2\gamma}} ~dy \\
	& \lesssim  \ln^{-\sigma}\left(\frac{d_{\Omega}}{\delta(x)}\right) \delta(x)^{\gamma} \int_{|z'| <1} \int_0^{\delta(x)/2} \frac{z_N^{2\gamma -2s}}{(|\delta(x)-z_N| + |z'|)^{N-2s+2\gamma}} ~dz_N dz'\\
	& \lesssim  \ln^{-\sigma}\left(\frac{d_{\Omega}}{\delta(x)}\right) \delta(x)^{\gamma} \int_0^{1/\delta(x)} h^{N-2} \int_0^{1/2} \frac{t^{2\gamma-2s}}{((1-t)+h)^{N-2s+2\gamma}} ~dt~dh\\
	& \lesssim  \ln^{-\sigma}\left(\frac{d_{\Omega}}{\delta(x)}\right) \delta(x)^{\gamma} \int_0^{1/\delta(x)} \frac{h^{N-2}}{(1+h)^{N-2s+2\gamma}} ~dh \\
	& \leq  C(s,\gamma,\eta,\phi) \ln^{1-\sigma}\left(\frac{d_{\Omega}}{\delta(x)}\right) \delta(x)^{\gamma}.
	\end{split}
	\end{equation*} \medskip
	
	\textit{Upper bound in $\mathscr{O}_4$.} We have
	\begin{equation*}
	\begin{split}
	\int_{\mathscr{O}_4}  \frac{G^{\Omega}(x,y)}{\delta(y)^{2s-\gamma}}& \ln^{-\sigma}\left(\frac{d_{\Omega}}{\delta(y)}\right)dy  \leq \delta(x)^{\gamma} \int_{\mathscr{O}_4} \frac{\delta(y)^{2\gamma-2s}}{|x-y|^{N-2s+2\gamma} \ln^{\sigma}\left(\frac{d_{\Omega}}{\delta(y)}\right)} ~dy \\
	& \lesssim \delta(x)^{\gamma} \int_{\{\frac{3\delta(x)}{2} < w_N < \eta\} \cap B(0,1)} \frac{w_N^{2\gamma -2s}}{(|\delta(x)-w_N| + |w'|)^{N-2s+2\gamma} \ln^{\sigma}\left(\frac{d_{\Omega}}{w_N}\right)} ~dw_N ~dw'\\
	& = \delta(x)^{\gamma} \int_{3/2}^{\eta/\delta(x)}  \int_0^{1/\delta(x)} \frac{t^{N-2} h^{2\gamma-2s}}{(|1-h|+t)^{N-2s+2\gamma} \ln^{\sigma}\left(\frac{d_{\Omega}}{h \delta(x)}\right) } ~dt~dh\\
	& =  \delta(x)^{\gamma} \int_{3/2}^{\eta/\delta(x)} \frac{h^{2\gamma-2s}}{(h-1)^{1-2s+2\gamma} \ln^{\sigma}\left(\frac{d_{\Omega}}{h \delta(x)}\right) } \int_0^{1/(h-1)\delta(x)} \frac{ r^{N-2} }{(1+r)^{N-2s+2\gamma} } ~dr~dh \\
	& \lesssim  \delta(x)^{\gamma} \int_{3/2}^{\eta/\delta(x)} \frac{h^{2\gamma-2s}}{(h-1)^{1-2s+2\gamma} \ln^{\sigma}\left(\frac{d_{\Omega}}{h \delta(x)}\right) } \int_1^{1/(h-1)\delta(x)} \frac{1 }{(1+r)^{2-2s+2\gamma} } ~dr~dh \\
	& \lesssim \delta(x)^{\gamma} \int_{3/2}^{\eta/\delta(x)} \frac{h^{2\gamma-2s}}{(h-1)^{1-2s+2\gamma} \ln^{\sigma}\left(\frac{d_{\Omega}}{h \delta(x)}\right) } ~dh \\
	& \lesssim \delta(x)^{\gamma} \int_{3/2}^{\eta/\delta(x)}  \frac{ \ln^{-\sigma}\left(\frac{d_{\Omega}}{h \delta(x)}\right) }{h} ~dh \\
	& = C(d_{\Omega}) \delta(x)^{\gamma} \int^{\ln(2d_{\Omega}/3\delta(x))}_{\ln(d_{\Omega}/\eta)} \frac{1}{t^\sigma} dt.
	\end{split}
	\end{equation*}
	This gives
	\begin{equation*}
	\int_{\mathscr{O}_4}  \frac{\ln^{-\sigma}\left(\frac{d_{\Omega}}{\delta(y)}\right)}{\delta(y)^{2s-\gamma}} G^\Omega(x,y) ~dy  \leq \frac{C(s,\gamma,\eta,\phi)}{1-\sigma}\ln^{1-\sigma}\left(\frac{d_{\Omega}}{\delta(x)}\right) \delta(x)^{\gamma}.
	\end{equation*} \medskip
	
	\textit{Upper bound in $\mathscr{O}_5$.} Again, by choosing $\eta$ small enough such that $ \eta < \frac{2d_{\Omega}}{\exp(1)}$,  for some $c \in (0,1)$, we have, for $y \in \mathscr{O}_5$,
	$$ \delta(y) \leq \frac{3}{2} \delta(x) \ \text{and} \ \ln^{-\sigma}\left(\frac{d_{\Omega}}{\delta(y)}\right) \leq c^{-\sigma} \ln^{-\sigma}\left(\frac{d_{\Omega}}{\delta(x)}\right) \leq c^{-1} \ln^{-\sigma}\left(\frac{d_{\Omega}}{\delta(x)}\right).$$
	Therefore
	\begin{equation*}
	\begin{split}
	\int_{\mathscr{O}_5}  \frac{G^{\Omega}(x,y)}{\delta(y)^{2s-\gamma}}& \ln^{-\sigma}\left(\frac{d_{\Omega}}{\delta(y)}\right)dy  \lesssim \delta(x)^{\gamma} \int_{\mathscr{O}_5} \frac{\delta(y)^{2\gamma-2s}}{|x-y|^{N-2s+2\gamma} \ln^{-\sigma}\left(\frac{d_{\Omega}}{\delta(y)}\right)} ~dy \\
	& \lesssim \delta(x)^{3\gamma-2s} \ln^{-\sigma}\left(\frac{d_{\Omega}}{\delta(x)}\right) \int_{\mathscr{O}_5} \frac{1}{|x-y|^{N-2s+2\gamma}} ~dy\\
	& \lesssim \delta(x)^{3\gamma-2s} \ln^{-\sigma}\left(\frac{d_{\Omega}}{\delta(x)}\right) \int_{\delta(x)/2}^{3 \delta(x)/2}  \int_{\delta(x)/2}^1 \frac{t^{N-2}}{(|\delta(x)-h|+t)^{N-2s+2\gamma}} ~dt~dh\\
	& \lesssim \delta(x)^{3\gamma-2s} \ln^{-\sigma}\left(\frac{d_{\Omega}}{\delta(x)}\right) \int_{\delta(x)/2}^1 t^{2s-2\gamma -1} \int_{-\delta(x)/2t}^{\delta(x)/2t}   \frac{1}{(|r|+1)^{N-2s+2\gamma}} ~dr~dt\\
	& \lesssim \delta(x)^{\gamma} \ln^{-\sigma}\left(\frac{d_{\Omega}}{\delta(x)}\right) \int_{1/2}^{1/\delta(x)} \rho^{2s-2\gamma-1} \int_{0}^{1/\rho} \frac{1}{(r+1)^{N-2s+2\gamma}} dr d\rho\\
	& \leq C(s,\gamma,\eta,\phi) \ln^{1-\sigma}\left(\frac{d_{\Omega}}{\delta(x)}\right) \delta(x)^{\gamma}.
	\end{split}
	\end{equation*}
Here in the last estimate we have used the assumption that $\gamma > s- \frac{1}{2}$.
	Finally, by collecting all the estimates in $\{\Omega_i\}$, we get the desired upper estimate.
\end{proof}

\begin{lemma}\label{lem:greest2}
Assume \eqref{G1}--\eqref{G2} hold and $s-\frac{1}{2} < \gamma < 2s$. Then for $\sigma \in [0,1)$, there holds
	$$\mathbb{G}^\Omega\left[ \frac{1}{\delta^{2s-\gamma}} \ln^{-\sigma}\left(\frac{d_{\Omega}}{\delta}\right) \right](x) \asymp \delta(x)^{\gamma} \ln^{1-\sigma}\left(\frac{d_{\Omega}}{\delta(x)}\right) \quad \forall \ x \in \Omega \setminus \Omega_{\frac{\eta}{2}},$$
	for some $\eta >0.$
\end{lemma}
\begin{proof} Let $\eta>0$ small and $I_1(x), I_2(x)$ as in \eqref{eq:GI1I2}.\vspace{0.1cm}\\
\textbf{Lower estimate.} We have
	\begin{equation*}
	\begin{split}
	 I_2(x)  & = \int_{\Omega \setminus \Omega_{\eta}} \frac{G^{\Omega}(x,y)}{\delta(y)^{2s-\gamma}} \ln^{-\sigma}\left(\frac{d_{\Omega}}{\delta(y)}\right) ~dy  \geq  \ln^{-\sigma}\left(\frac{d_{\Omega}}{\eta}\right)\mathbb{G}^\Omega\left[ \frac{{\bf 1}_{\Omega \setminus \Omega_{\eta}}}{\delta^{2s-\gamma}}\right](x)\\
	& \geq C \ln^{-\sigma}\left(\frac{2d_{\Omega}}{\eta} \right) \delta(x)^{\gamma} \geq C(\eta) \ln^{1-\sigma}\left(\frac{2d_{\Omega}}{\eta} \right) \delta(x)^{\gamma} \\
	&\geq C(\eta) \ln^{1-\sigma}\left(\frac{d_{\Omega}}{\delta(x)} \right) \delta(x)^{\gamma}.
	\end{split}
	\end{equation*} \medskip
	
	\textbf{Upper estimate.} Using the estimates from \cite[Proof of Lemma 3.2, Page 37]{AbaGomVaz_2019}, for $\eta$ small enough, we obtain
	\begin{equation*}
	\begin{split}
	&I_1(x) + I_2(x) \\
	&\leq \ln^{-\sigma}\left(\frac{d_{\Omega}}{\eta}\right) \delta(x)^{\gamma} \left( \int_{\Omega_{\frac{\eta}{4}}}\frac{G^{\Omega}(x,y)}{\delta(y)^{2s-\gamma}}~dy + \int_{\Omega_{\eta} \setminus \Omega_{\frac{\eta}{4}}} \frac{G^{\Omega}(x,y)}{\delta(y)^{2s-\gamma}}~dy \right) + \frac{\ln^{-\sigma}\left(2\right)}{\eta^{2s-\gamma}} \mathbb{G}^\Omega[{\bf 1}_\Omega](x)\\
	& \leq C \ln^{-\sigma} \left(\frac{d_{\Omega}}{\eta}\right) \delta(x)^{\gamma} \left( 1 + \frac{1}{\eta^{2s-\gamma}}\right) + C(d_{\Omega}) \ln^{1-\sigma}\left(\frac{d_{\Omega}}{\delta(x)} \right) \delta(x)^{\gamma} \frac{1}{\eta^{2s-\gamma} \ln^{1-\sigma}\left(\frac{2d_{\Omega}}{\eta}\right)} \\
	& \leq C(d_{\Omega},\eta) \ln^{1-\sigma}\left(\frac{d_{\Omega}}{\delta(x)} \right) \delta(x)^{\gamma}.
	\end{split}
	\end{equation*}
	This implies the desired upper bound.
\end{proof}
\subsection{Existence and uniqueness results}
\noindent For $\epsilon>0$, consider the following approximating problem
\begin{equation}\label{sing:approx} \tag{$P_s^\epsilon$} \left\{
\begin{aligned}
\mathbb{L} u & =\frac{1}{(u_\epsilon+\epsilon)^q}\quad &&\text{ in }  \Omega, \\
u &>0 &&\text{ in } \Omega, \\
u & = 0 && \text { on } \partial\Omega \text{ or in } \Omega^c \text{ if applicable},
\end{aligned}
\right.
\end{equation}
where $q>0$.
\begin{proposition}\label{pro:app}
	Assume \eqref{L1}--\eqref{L4} and \eqref{G1}--\eqref{G3} hold. Then for any $\epsilon>0$, there exists a unique weak-dual solution $u_\epsilon \in L^1_0(\Omega, \delta^\gamma) \cap L^\infty(\Omega)$ of problem \eqref{sing:approx} in the sense that
	\begin{equation}\label{notion:approx}
	\int_{\Omega} u_\epsilon \xi ~dx = \int_{\Omega} \frac{1}{(u_\epsilon+\epsilon)^q} \mathbb{G}^\Omega[\xi] ~dx, \quad \forall \ \xi \in
	\delta^\gamma L^\infty(\Omega).
	\end{equation}
	The solution is represented by
	$$u_\epsilon= \mathbb{G}^\Omega\left[\frac{1}{(u_\epsilon+\epsilon)^q}\right] \quad \text{a.e. in } \Omega. $$
	Moreover, the mapping $\epsilon \mapsto u_\epsilon$ is decreasing.
\end{proposition}
\begin{proof}
	Fix $\epsilon>0$. For any $\varphi \in L^\infty(\Omega)$, in light of  the continuous embedding property of the Green operator $\mathbb{G}^\Omega : L^\infty(\Omega) \to L^\infty(\Omega)$ (see \cite[Theorem 2.1]{AbaGomVaz_2019}) and Lemma \ref{IBP}, the function
	$$ u = \mathbb{G}^\Omega\left[\frac{1}{(\varphi^++\epsilon)^q}\right] \in L^\infty(\Omega),
	$$
	where $\varphi^+= \max\{\varphi,0\}$, is the unique positive weak-dual solution  the following problem
	\begin{equation}\label{sing:approx1}
	 \left\{
	\begin{aligned}
	\mathbb{L} u & =\frac{1}{(\varphi^++\epsilon)^q} \quad &&\text{ in }  \Omega, \\
	u &>0 &&\text{ in } \Omega, \\
	u & = 0 && \text { on } \partial\Omega \text{ or in } \Omega^c \text{ if applicable},
	\end{aligned}
	\right.
	\end{equation}
	 in the sense that
	\begin{equation} \label{wdsol-e} \int_{\Omega} u \xi ~dx = \int_{\Omega} \frac{1}{(\varphi^++\epsilon)^q} \mathbb{G}^\Omega[\xi] ~dx, \quad \forall \ \xi \in
	\delta^\gamma L^\infty(\Omega).
	\end{equation}
	Define the solution map $\mathbb{T}_{\epsilon}:  L^\infty(\Omega) \to  L^\infty(\Omega)$ as
	$$ \mathbb{T}_{\epsilon}(\varphi):= \mathbb{G}^\Omega\left[\frac{1}{(\varphi^++\epsilon)^q}\right], \quad \varphi \in L^\infty(\Omega).$$
	We will use the Schauder fixed point theorem to show that $\mathbb{T}_{\epsilon}$ admits a fixed point which is a weak-dual solution of \eqref{sing:approx1}.
	
	For any $\varphi \in L^\infty(\Omega)$ and $x \in \Omega$, we obtain
	\begin{equation*}
	\left|\mathbb{G}^\Omega\left[\frac{1}{(\varphi^++\epsilon)^q}\right] (x)\right| \leq \epsilon^{-q}\mathbb{G}^\Omega[1](x) \leq \epsilon^{-q}\| \mathbb{G}^\Omega[1] \|_{L^\infty(\Omega)}=:R_\epsilon,
	\end{equation*}
	and thus,
	\begin{equation}\label{eq:invar}
	\|\mathbb{T}_{\epsilon}(\varphi)\|_{L^\infty(\Omega)} \leq R_\epsilon.
	\end{equation}
	Put
	$$ \mathscr{D}_\epsilon:= \{\varphi \in L^\infty(\Omega): \|\varphi\|_{L^\infty(\Omega)} \leq R_\epsilon\}$$
	then $\mathscr{D}_{\epsilon}$ is a closed, convex subset of $L^\infty(\Omega)$. Moreover, by \eqref{eq:invar}, $\mathbb{T}_{\epsilon}(\mathscr{D}_{\epsilon}) \subset \mathscr{D}_{\epsilon}$. \medskip
	
\noindent \textbf{Claim 1:} $\mathbb{T}_{\epsilon}$ is continuous.

Indeed, let $\{\varphi_k\} \subset L^\infty(\Omega)$ such that $\varphi_k \to \varphi$ in $L^\infty(\Omega)$. Put $u_k=\mathbb{T}_{\epsilon}(\varphi_k)$ and  $u=\mathbb{T}_{\epsilon}(\varphi)$.
	For any $x \in \Omega$, we have
	\begin{equation*}
	\begin{split}
	 |u_k(x)-u(x)|&=  \left| \mathbb{G}^\Omega\left[\frac{1}{(\varphi_k^++\epsilon)^q}\right] (x)- \mathbb{G}^\Omega\left[\frac{1}{(\varphi^++\epsilon)^q}\right] (x)\right| \\
	&=  \int_\Omega G^{\Omega}(x,y) \left|\frac{1}{(\varphi_k^++\epsilon)^q}- \frac{1}{(\varphi^++\epsilon)^q} \right| ~dy \\
	&\leq \frac{q}{ \epsilon^{q+1}}\int_\Omega G^{\Omega}(x,y) \left|\varphi_k^+(y)-\varphi^+(y)\right| ~dy \\
	&\leq \frac{q}{ \epsilon^{q+1}} \mathbb{G}^{\Omega}\left[|\varphi_k^+-\varphi^+|\right](x)\\
	&\leq \frac{q}{ \epsilon^{q+1}}\| \mathbb{G}^{\Omega}[1] \|_{L^\infty(\Omega)} \|\varphi_k - \varphi \|_{L^\infty(\Omega)}.
	\end{split}
	\end{equation*}
	Therefore
	\begin{equation} \label{est:wkw} \| u_k - u \|_{L^\infty(\Omega)} \leq \frac{q}{ \epsilon^{q+1}} \| \mathbb{G}^{\Omega}[1] \|_{L^\infty(\Omega)} \|\varphi_k - \varphi \|_{L^\infty(\Omega)}.
	\end{equation}
	Consequently, since $\varphi_k \to \varphi$ in $L^\infty(\Omega)$, by letting $k \to \infty$ in \eqref{est:wkw}, we derive that $u_k \to u$ in $L^\infty(\Omega)$. Thus $\mathbb{T}_{\epsilon}$ is continuous. \medskip
	
\noindent \textbf{Claim 2:} $\mathbb{T}_{\epsilon}(\mathscr{D}_{\epsilon})$ is relatively compact.

Indeed, by Arzela-Ascoli theorem and \eqref{eq:invar}, it is enough to show that the set $\mathbb{T}_{\epsilon}(\mathscr{D}_{\epsilon})$ is equicontinuous, namely for any $x \in \Omega$ and $\mu >0$, there exists $\nu>0$ such that if for every $y \in B(x,\nu)$ and $u \in \mathbb{T}_{\epsilon}(\mathscr{D}_{\epsilon})$ there holds
	$$ |u(x) - u(y)| < \mu.
	$$
	To this end, let $x \in \Omega$ and $\mu >0$. Using the fact that $\mathbb{G}^{\Omega}[L^\infty(\Omega)] \subset C(\overline{\Omega})$ is continuous (\cite[Theorem 2.10]{AbaGomVaz_2019}),  we get $\mathbb{G}^{\Omega}[1] \in C(\overline{\Omega})$. This implies
	$$ \lim_{y \to x}\int_{\Omega}G^{\Omega}(y,z)dz = \int_{\Omega}G^{\Omega}(x,z)dz.
	$$
	This, together with assumption \eqref{G3}, implies that
	$$ \lim_{y \to x}\int_{\Omega}|G^{\Omega}(x,z)-G^{\Omega}(y,z)|dz = 0.
	$$
	Therefore there exists $\nu>0$ such that if $y \in B(x,\nu)$ then
	$$ \int_{\Omega}|G^{\Omega}(x,z)-G^{\Omega}(y,z)|dz < \mu \epsilon^q.
	$$
	Now take $u \in \mathbb{T}_{\epsilon}(\mathscr{D}_{\epsilon})$, then there exists $\varphi \in \mathscr{D}_\epsilon$ such that $u=\mathbb{T}_{\epsilon}(\varphi)$. For $y \in B(x,\nu)$, we have
	\begin{equation*}
	\begin{split}
	|u(x)-u(y)|&=  \left| \mathbb{G}^\Omega\left[\frac{1}{(\varphi^++\epsilon)^q}\right] (x)- \mathbb{G}^\Omega\left[\frac{1}{(\varphi^++\epsilon)^q}\right] (y)\right| \\
	& \leq \frac{1}{\epsilon^q}  \int_\Omega |G^{\Omega}(x,z)- G^{\Omega}(y,z)|dz  \leq  \mu.
	\end{split}
	\end{equation*}
	Therefore $\mathbb{T}_{\epsilon}(\mathscr{D}_{\epsilon})$ is equicontinuous. By Arzela-Ascoli theorem, $\mathbb{T}_{\epsilon}(\mathscr{D}_{\epsilon})$ is relatively compact. Therefore, we have proved Claim 2.

	We infer from the Schauder fixed point theorem that there exists a fixed point $u_{\epsilon} \in \mathscr{D}_{\epsilon}$ of $\mathbb{T}_{\epsilon}$, namely
	$$ 0< u_\epsilon = \mathbb{T}_{\epsilon}(u_\epsilon) = \mathbb{G}^\Omega\left[\frac{1}{(u_\epsilon^++\epsilon)^q}\right] = \mathbb{G}^\Omega\left[\frac{1}{(u_\epsilon+\epsilon)^q}\right].
	$$
	From Lemma \ref{IBP}, we see that  $u_\epsilon$ is a weak-dual solution of \eqref{sing:approx1} in the sense of \eqref{wdsol-e}.
	
	Finally, we will show that $\epsilon \mapsto u_\epsilon$ is decreasing. To this end, let $ \epsilon > \epsilon'$ and denote by $u_{\epsilon}$ and $u_{\epsilon'}$ are the corresponding weak-dual solutions of \eqref{sing:approx}.  For any $\xi \in C_c^\infty(\Omega)$, $\xi \geq 0$, $\mathbb{G}^\Omega[\xi] > 0 $,  from Kato's inequality (see, Lemma \ref{Kato}), we obtain
	\begin{equation}
	0 \leq \int_\Omega (u_{\epsilon}- u_{\epsilon'})^+ \xi ~dx \leq \int_{ \{ u_{\epsilon} \geq u_{\epsilon'} \} } \left( \frac{1}{(u_{\epsilon}+\epsilon)^q}- \frac{1}{(u_{\epsilon'}+\epsilon')^q} \right) \mathbb{G}^\Omega[\xi]  ~dx \leq 0,
	\end{equation}
	which  implies that $(u_{\epsilon} - u_{\epsilon'})^+ = 0$ in $\Omega$, whence $u_{\epsilon} \leq u_{\epsilon'}$ in $\Omega$.
\end{proof}

Let $\varphi_1$ and $\lambda_1$ are the first eigenfunction and eigenvalue for the operator $\mathbb{L}$, then by \cite[Proposition 3.7]{ChaGomVaz_2020} and \cite[Proposition 5.3, 5.4]{BonFigVaz_2018}, we have
\begin{equation}\label{est:eigen}
\varphi_1 \asymp \delta^\gamma \ \text{and} \ \varphi_1= \lambda_1 \mathbb{G}^\Omega[\varphi_1] \quad \text{in} \ \Omega
\end{equation}
\begin{theorem}\label{thm:pure:sing:exist}
Assume \eqref{L1}--\eqref{L4} and \eqref{G1}--\eqref{G3} hold and $(q, \gamma) \in \mathbb{E}$. Then there exists a unique weak-dual solution $u_*$ of the problem \eqref{sing:pure} in the sense of definition \ref{def:weak-dual}.
\end{theorem}
\begin{proof}
	Let $u_1$ be the weak-dual solution of $(P_s^1)$, then $u_1$ can be represented as
	\begin{equation} \label{u1} u_1= \mathbb{G}^\Omega\left[\frac{1}{(u_1+1)^q}\right] \quad \text{in } \Omega.
	\end{equation}
	On the one hand, we see that
	$$ 0< u_1 \leq \mathbb{G}^\Omega[1] \quad \text{in } \Omega.
	$$
	On the other hand, for $x \in \Omega$,
	\begin{equation}\label{est:appr:solu:small:lower}
	\begin{split}
	\varphi_1(x) &= \lambda_1 \mathbb{G}^\Omega[\varphi_1](x) = \lambda_1 \mathbb{G}^\Omega\left[\varphi_1 \frac{(\|u_1\|_{L^\infty(\Omega)}+1)^q}{(\|u_1\|_{L^\infty(\Omega)}+1)^q}\right](x)\\
	& \leq \lambda_1 \|\varphi_1\|_{L^\infty(\Omega)} (\|u_1\|_{L^\infty(\Omega)}+1)^q \mathbb{G}^\Omega\left[ (\|u_1\|_{L^\infty(\Omega)}+1)^{-q}\right](x)\\
	& \leq C \mathbb{G}^\Omega\left[ \frac{1}{(u_1+1)^q}\right](x) = C u_1(x),
	\end{split}
	\end{equation}
	where $C$ depends on $N,s,\gamma,\Omega,q,\lambda_1, \phi_1$.
	
	For $0<\epsilon<1$, let $u_\epsilon$ be the weak-dual solution of \eqref{sing:approx}, then $u_\epsilon$ is represented as
	$$u_\epsilon(x)= \mathbb{G}^\Omega\left[\frac{1}{(u_\epsilon+\epsilon)^q}\right](x) = \int_{\Omega}  \frac{G^{\Omega}(x,y)}{(u_\epsilon(y)+\epsilon)^q} ~dy, \quad x \in \Omega. $$
	Combining the fact that $\epsilon \mapsto u_\epsilon$ is decreasing, estimate \eqref{est:eigen} and \eqref{est:appr:solu:small:lower}, we deduce that, for $\epsilon<1$,
	\begin{equation} \label{compare-ep1} u_\epsilon \geq u_1 \geq C\varphi_1 \geq C\delta^\gamma \quad \text{in } \Omega,
	\end{equation}
	where the constant $C$ is independent of $\epsilon$.
	
	To prove the upper boundary behavior, we consider the following cases. \medskip
	
\noindent\textbf{Case 1:} $0< q < q^{*}_{s,\gamma}$. Using estimate \eqref{compare-ep1} and Lemma \ref{est:green}, we get
	\begin{equation}\label{est:appr:solu:small:upper}
	\begin{split}
	u_\epsilon & \lesssim \mathbb{G}^{\Omega}\left[\frac{1}{(\delta^{\gamma}+\epsilon)^q}\right]
	 \leq  \mathbb{G}^{\Omega}\left[\frac{1}{\delta^{\gamma q}}\right] \lesssim \delta^\gamma \quad \text{if} \ q < \min\left\{q^{*}_{s,\gamma}, 1+ \frac{1}{\gamma}\right\}.
	\end{split}
	\end{equation}
	Combining \eqref{compare-ep1}-\eqref{est:appr:solu:small:upper}, there exist $C_1, C_2>0$ such that for any $0<\epsilon<1$,
	\begin{equation}\label{est:appr:solu:small:both}
	C_1 \delta^\gamma \leq u_\epsilon \leq C_2 \delta^\gamma  \quad \text{if} \  q < \min\left\{q^{*}_{s,\gamma}, 1+ \frac{1}{\gamma}\right\}.
	\end{equation}
We note that if $\gamma \geq s- \frac{1}{2}$ then $q^{*}_{s,\gamma} \leq 1 + \frac{1}{\gamma}$, hence the condition on $q$ becomes $q < q^{*}_{s,\gamma}$.
	
\noindent	\textbf{Case 2: } $q \in (q^{*}_{s,\gamma}, q^{\ast \ast}_{s, \gamma})$. Put $\epsilon_1= \epsilon^{\frac{1}{\alpha}}>0$. Let $w_{\epsilon_1}$ be the weak-dual solution of
	\begin{equation} \label{Pew}
	\tag{$P_w^\epsilon$} \left\{
	\begin{aligned}
	\mathbb{L} w_{\epsilon_1} & =\frac{1}{(\delta+\epsilon_1)^\beta} \quad &&\text{ in }  \Omega, \\
	w_{\epsilon_1} &>0 &&\text{ in } \Omega, \\
	w_{\epsilon_1} & = 0 && \text { on } \partial\Omega \text{ or in } \Omega^c \text{ if applicable}.
	\end{aligned}
	\right.
	\end{equation}
	 Using Lemma \ref{est:green} with $\beta$ and $\alpha$ are defined in \eqref{alphabeta}, we obtain
	\begin{equation}\label{est:weight:approx1}
	\begin{split}
	w_{\epsilon_1}= \mathbb{G}^\Omega\left[\frac{1}{(\delta+\epsilon_1)^\beta}\right] \leq \mathbb{G}^\Omega\left[\frac{1}{\delta^\beta}\right] \leq C \delta^{\alpha} \quad \text{in }  \Omega \quad \text{if} \ \gamma > \max\left\{ \alpha, s-\frac{1}{2}, \beta-1 \right\}.
	\end{split}
	\end{equation}
On the other hand for $0< \eta <1$, Lemma \ref{lem:greenest} gives
	\begin{equation}\label{est:weight:approx2}
	\begin{split}
	w_{\epsilon_1}&(x)= \mathbb{G}^\Omega\left[\frac{1}{(\delta+\epsilon_1)^\beta}\right](x) \\
	&
	\begin{aligned}
	\geq  \left\{
	\begin{array}{ll}
	\displaystyle{\mathbb{G}^{\Omega}\left[\frac{1}{(\delta+\epsilon_1)^\beta} {\bf 1}_{\Omega_{\eta}}\right](x)} \geq  c_1 \left(\frac{1}{2}(\delta(x)+ \epsilon_1)^\alpha - \epsilon \right) &  \text{ if } x \in \Omega_{\frac{\eta}{2}},\\
	c_2 \ (\delta(x)+ \epsilon_1)^\alpha - \epsilon &  \text{ if } x \in \Omega \setminus \Omega_{\frac{\eta}{2}}, \\
	\end{array}
	\right.
	\end{aligned}\\
	& \geq \min\{c_1,c_2\} \left( \frac{1}{2}(\delta(x)+ \epsilon_1)^\alpha - \epsilon\right),   \quad x \in \Omega \quad \text{if} \ \gamma > \max\left\{ \alpha, s-\frac{1}{2}, \beta-1 \right\}.\\
	\end{split}
	\end{equation}
	Combining \eqref{est:weight:approx1}-\eqref{est:weight:approx2}, there exist constants $c_3, c_4>0$ depending upon $\eta, \alpha,s$ (but independent of $\epsilon_1$) such that
	\begin{equation}\label{est:weight:approx3}
	c_3 \left(\frac{1}{2}(\delta(x)+ \epsilon_1)^\alpha - \epsilon\right)   \leq w_{\epsilon_1}(x) \leq c_4 \delta(x)^\alpha, \ x \in \Omega \quad \text{if} \ \gamma > \max\left\{ \alpha, s-\frac{1}{2}, \beta-1 \right\}.
	\end{equation}
	
\noindent Define
	\begin{equation}
	\underline u_{\epsilon } := c_\eta w_{\epsilon_1} \ \text{with} \ 0< c_\eta < \frac{C_1}{c_4} \left(\frac{\eta}{2}\right)^{\gamma-\alpha},
	\end{equation}
	where $C_1, c_4$ are the constants defined in \eqref{est:appr:solu:small:both} and \eqref{est:weight:approx3} respectively. We note that $c_\eta$ is independent of $\epsilon$ such that
	$$ \underline u_{\epsilon } \leq u_\epsilon \ \text{in } \Omega \setminus \Omega_{\frac{\eta}{2}}
	$$
	and
	$$
	\underline u_{\epsilon } + \epsilon \leq 2 c_4 c_\eta (\delta + \epsilon_1)^\alpha + (1-c_4 c_\eta) \epsilon \quad \text{ in } \Omega.
	$$
	
\noindent If $2 c_4 c_\eta (\delta(x)+ \epsilon_1)^\alpha \geq (1-c_4 c_\eta) \epsilon$ then by choosing $\eta$ small enough such that $c_\eta < (4c_4)^{-\frac{q}{q+1}}$, we have
\[ \left(\underline u_{\epsilon } + \epsilon\right)^{-q} \geq (4c_4 c_\eta)^{-q} (\delta+ \epsilon_1)^{-\alpha q} > c_\eta  (\delta+ \epsilon_1)^{-\alpha q} = c_\eta  (\delta+ \epsilon_1)^{-\beta}= \mathbb{L} \underline u_{\epsilon } \quad  \text{in} \ \Omega.
\]
If $2c_4 c_\eta (\delta(x)+ \epsilon_1)^\alpha \leq (1-c_4 c_\eta) \epsilon$ then again by choosing $\eta$ small enough such that $c_4 c_\eta <1$ and $c_\eta < (2(1-c_4 c_\eta))^{-q} $, we have
\[
\begin{split}
    \left(\underline u_{\epsilon } + \epsilon\right)^{-q} \geq (2(1-c_4 c_\eta))^{-q} \epsilon^{-q} & \geq (2(1-c_4 c_\eta))^{-q}  (\delta+ \epsilon_1)^{-\alpha q} \\
    & > c_\eta  (\delta + \epsilon_1)^{-\alpha q} = c_\eta  (\delta + \epsilon_1)^{-\beta} = \mathbb{L} \underline u_{\epsilon } \quad  \text{in} \ \Omega.
\end{split}
\]
Combining the above cases, we obtain
$$ \mathbb{L} \underline u_\epsilon \leq \frac{1}{(\underline u_\epsilon+\epsilon)^q} \quad \text{in } \Omega.
$$
Recalling that $u_\epsilon$ is the weak-dual solution of \eqref{sing:approx}. By applying the Kato type inequality \eqref{kato+}, we get $\underline u_{\epsilon} \leq u_\epsilon$ in $\Omega$, namely there exist constants $0 < c_5, c_6 < \frac{1}{2}$ (by taking $\eta$ small enough) such that
	\begin{equation}\label{est:appr:solu:large:lower}
	c_5 (\delta + \epsilon_1)^\alpha - c_6 \epsilon \leq u_\epsilon \quad \text{in } \Omega.
	\end{equation}
Using the integral representation of the solution $u_\epsilon$ and \eqref{est:weight:approx1} with $\beta = \alpha q = \frac{2s q}{q+1}$, we obtain the upper bound as follows
	\begin{equation}\label{est:appr:solu:large:upper}
	\begin{split}
	u_\epsilon &=\mathbb{G}^{\Omega}\left[ \frac{1}{(u_\epsilon+1)^q} \right] \leq  \mathbb{G}^{\Omega}\left[\frac{1}{(c_5 (\delta+ \epsilon_1)^\alpha +  (1- c_6) \epsilon)^q}\right] \\
	& \leq  \mathbb{G}^{\Omega}\left[\frac{1}{\delta^{\alpha q}}\right] = \mathbb{G}^{\Omega}\left[\frac{1}{\delta^{\beta}}\right] \lesssim  \delta^\alpha \quad \text{if} \ \gamma > \max\left\{ \alpha, s-\frac{1}{2}, \beta-1 \right\}.
	\end{split}
\end{equation}
The above condition breakdown the range of $q$ as
	\[ q \in \left(q^{*}_{s,\gamma}, \infty\right) \ \text{if} \ 2s-1 \leq \gamma
	\]
	and
	\[q \in \left(q^{*}_{s,\gamma}, \frac{\gamma+1}{2s-\gamma -1}\right) \ \text{if} \ \gamma <2s-1.\]
\textbf{Case 3:} $q = q^{*}_{s,\gamma}$. In this case, by using \eqref{est:appr:solu:small:lower} and Lemma \ref{est:green}, we get
\begin{equation}\label{est:appr:solu:upper}
	\begin{aligned}
	u_\epsilon = \mathbb{G}^{\Omega}\left[\frac{1}{(u_\epsilon+ \epsilon)^q}\right] &\leq C  \mathbb{G}^{\Omega}\left[\frac{1}{(\delta^\gamma+\epsilon)^q}\right] \\
	&\leq C   \mathbb{G}^{\Omega}\left[\frac{1}{\delta^{\gamma q}}\right] \leq C  \delta^\gamma  \ln\left(\frac{d_{\Omega}}{\delta}\right) \quad \text{if} \ q = q^{*}_{s,\gamma} < 1+ \frac{1}{\gamma}.
	\end{aligned}
\end{equation}
The above condition $q^\ast_{s, \gamma} < 1+ \frac{1}{\gamma}$ follows from the assumption $\gamma > s-\frac{1}{2}.$

Define $u_*:= \lim_{\epsilon \to 0} u_\epsilon$. Then, by using uniform  boundary behavior in $\epsilon$ and Lebesgue dominated convergence theorem, we pass to the limits in \eqref{notion:approx} and obtain
	\begin{equation}
	\int_{\Omega} u_* \xi ~dx = \int_{\Omega} \frac{1}{u_*^q} \mathbb{G}^\Omega[\xi] ~dx \quad \forall \ \xi \in \delta^\gamma L^\infty(\Omega).
	\end{equation}
	Again, the uniqueness of the weak-dual solution from Kato type inequality \eqref{kato+}. The proof is complete.
\end{proof}
\subsection{Boundary behavior of weak-dual solutions}
\begin{theorem}\label{thm:pure:sing:boundbeha}
Assume \eqref{L1}--\eqref{L4} and \eqref{G1}--\eqref{G3} hold and $(q, \gamma) \in \mathbb{E}$. Let $u_*$ be the weak-dual solution of problem \eqref{sing:pure} obtained in Theorem \ref{thm:pure:sing:exist}. Then $u_* \in \mathcal{A}_{\ast}(\Omega)$ where $\mathcal{A}_{\ast}(\Omega)$ is defined in \eqref{puresing:bdrybeh}.
\end{theorem}
\begin{proof}
	Let $u_*$ be the weak-dual solution of the problem \eqref{sing:pure}. Using \eqref{est:appr:solu:small:both} for $q<q^{*}_{s,\gamma}$ and  \eqref{est:appr:solu:large:lower}, \eqref{est:appr:solu:large:upper} for $q>q^{*}_{s,\gamma}$ and \eqref{est:appr:solu:small:lower},  \eqref{est:appr:solu:upper} for $q=q^{*}_{s,\gamma}$ and $\gamma<s$, we get $u_* \in \mathcal{A}_{\ast}(\Omega)$. Now, it remains to prove the optimal boundary behavior when $q= q^{*}_{s,\gamma}$ and $ s \leq \gamma < 2s$. \smallskip
	
	\textbf{Case 1:} $ s < \gamma < 2s$.
	For $\sigma \in (0,1)$, define
	\begin{equation} \label{phisigma}
	\begin{aligned}
	\psi_\sigma(x):=  \left\{
	\begin{array}{ll}
	\displaystyle{\mathbb{G}^\Omega\left[\frac{1}{\delta^{2s-\gamma}}\ln^{-\sigma}\left(\frac{d_{\Omega}}{\delta}\right) \right](x)} &  \text{ if } x \in \Omega, \\
	0
	& \text{ if } x \in \partial\Omega\, \ \mbox{or $\Omega^c$ if applicable}.\\
	\end{array}
	\right.
	\end{aligned}
	\end{equation}
	Since $\frac{1}{\delta^{2s-\gamma}}\ln^{-\sigma}(\frac{d_{\Omega}}{\delta})  \in L^1(\Omega, \delta^\gamma)$, by using Lemma \ref{IBP}, Lemma \ref{lem:greest1} and Lemma \ref{lem:greest2}, we obtain $\psi_\sigma \in L^1(\Omega, \delta^
	\gamma)$,
	\begin{equation}\label{est:grecric}
	\frac{1}{M}  \delta(x)^{\gamma} \ln^{1-\sigma}\left(\frac{d_{\Omega}}{\delta(x)}\right) \leq  \psi_\sigma(x) \leq M \delta(x)^{\gamma} \ln^{1-\sigma}\left(\frac{d_{\Omega}}{\delta(x)}\right), \ x \in \Omega,
	\end{equation}
	where $M$ is independent of the parameter $\sigma \in (0, \sigma_0)$ for some $0< \sigma_0 <1$ and
	\begin{align*} \int_{\Omega} \psi_\sigma \xi ~dx &= \int_{\Omega} \mathbb{G}^\Omega\left[ \frac{1}{\delta^{2s-\gamma}}\ln^{-\sigma}\left(\frac{d_{\Omega}}{\delta}\right) \right] \xi ~dx \\
	&= \int_{\Omega} \frac{1}{\delta^{2s-\gamma}}\ln^{-\sigma}\left(\frac{d_{\Omega}}{\delta}\right) \mathbb{G}^\Omega[\xi]  ~dx, \quad \forall \ \xi \in  \delta^\gamma L^\infty(\Omega).
	\end{align*}
	Let $u_*$ be the weak-dual solution of \eqref{sing:pure}. For $\sigma=q \in (0,1)$, define $\underline \psi_\sigma:= C^{-q} \psi_\sigma$ where $C$ is defined in \eqref{est:appr:solu:upper}. Then using the upper bound of $u$ in \eqref{est:appr:solu:upper}, we get
	\begin{equation}
	\mathbb{L} \underline \psi_\sigma = \frac{1}{C^q} \frac{1}{\delta^{2s-\gamma} \ln^{q}\left(\frac{d_{\Omega}}{\delta}\right)} \leq \frac{1}{u_*^q} = \mathbb{L} u_*.
	\end{equation}
	From \eqref{est:grecric} and Kato's inequality, we infer that
	\begin{equation}\label{est:grecric1}
	\frac{1}{M C^q}  \delta(x)^{\gamma} \ln^{1-q}\left(\frac{d_{\Omega}}{\delta(x)}\right) \leq C^{-q} \psi_\sigma(x) \leq u_*(x), \ \ x \in \Omega.
	\end{equation}
	Define $\overline \psi_{\sigma_1}:= M^q C^{q^2} \psi_{\sigma_1}$  for $\sigma_1= q(1-q) \in (0,1)$ and using \eqref{est:grecric1} and the fact that $q \gamma = 2s - \gamma$, we obtain
	\begin{equation}
	\mathbb{L} u_* = \frac{1}{u_*^q} \leq \frac{M^q C^{q^2}}{ \delta^{2s-\gamma} \ln^{q(1-q)}\left(\frac{d_{\Omega}}{\delta}\right)} = \mathbb{L} \overline \psi_{\sigma_1}.
	\end{equation}
	Then by again using \eqref{est:grecric} and the Kato inequality, we get
	\begin{equation*}
	u_* \leq  M^q C^{q^2} \psi_{\sigma_1} \leq M^{q+1} C^{q^2}  \delta^{\gamma} \ln^{1-q+q^2}\left(\frac{d_{\Omega}}{\delta}\right) \quad \text{in } \Omega.
	\end{equation*}
	Now, by iterating theses estimates, we get for any $m \in \mathbb{N}$,
	\begin{equation*}
	\begin{split}
	\frac{1}{M^{1+q+q^2 + \dots + q^{2m-2}} C^{q^{2m-1}}} &\delta^{\gamma} \ln^{1-q+q^2+ \dots - q^{2m+1}}\left(\frac{d_{\Omega}}{\delta}\right)  \leq u_*\\
	&\leq M^{1+q+q^2 + \dots + q^{2m-1}} C^{q^{2m}} \delta^{\gamma} \ln^{1-q+q^2 + \dots + q^{2m+2}}\left(\frac{d_{\Omega}}{\delta}\right) \quad \text{in } \Omega.
	\end{split}
	\end{equation*}
	Passing $m \to \infty$ and recalling that $q= q^{*}_{s,\gamma} \in (0,1)$, we obtain
	\begin{equation}
	\frac{1}{M^{\frac{\gamma}{2 \gamma-2s}}} \delta^\gamma  \ln^{\frac{\gamma}{2s}}\left(\frac{d_{\Omega}}{\delta}\right) \leq u_* \leq M^{\frac{\gamma}{2 \gamma-2s}} \delta^\gamma \ln^{\frac{\gamma}{2s}}\left(\frac{d_{\Omega}}{\delta}\right) \quad \text{in } \Omega.
	\end{equation}
	
	\textbf{Case 2:} $\gamma=s$. In this case, $q=q_{s,\gamma}^*=1$ and hence, by taking $\sigma= \frac{1}{2}$ in \eqref{phisigma}, we derive that $\psi_{\frac{1}{2}}$ satisfies
\begin{equation} \label{psi1/2} \psi_{\frac{1}{2}} = \mathbb{G}^{\Omega} \left[ \frac{1}{\delta^\gamma \ln^{\frac{1}{2}}(\frac{d_{\Omega}}{\delta})} \right] \quad \text{in } \Omega.
\end{equation}
Next, by taking $\sigma=\frac{1}{2}$ and using \eqref{est:grecric}, we have
\begin{equation}\label{est:1/2}
	\frac{1}{M}  \delta(x)^{\gamma} \ln^{\frac{1}{2}}\left(\frac{d_{\Omega}}{\delta(x)}\right) \leq  \psi_{\frac{1}{2}}(x) \leq M \delta(x)^{\gamma} \ln^{\frac{1}{2}}\left(\frac{d_{\Omega}}{\delta(x)}\right), \ x \in \Omega.
\end{equation}
Recall that
\begin{equation} \label{1/u} u_*= \mathbb{G}^{\Omega} \left[\frac{1}{u_*}\right] \quad \text{in } \Omega.
\end{equation}
Combining \eqref{psi1/2}--\eqref{1/u} and applying Kato type inequality \eqref{kato+} for the functions $\frac{u_*}{M}-\psi_{\frac{1}{2}}$ and $\psi_{\frac{1}{2}} - Mu_*$ successively, we obtain
	$$ \frac{1}{M^2}  \delta(x)^{\gamma} \ln^{\frac{1}{2}}\left(\frac{d_{\Omega}}{\delta(x)}\right) \leq  \frac{1}{M}\psi_{\frac{1}{2}} \leq u_*(x) \leq M \psi_{\frac{1}{2}} \leq M^2 \delta(x)^{\gamma} \ln^{\frac{1}{2}}\left(\frac{d_{\Omega}}{\delta(x)}\right), \ x \in \Omega.$$
The proof is complete.
\end{proof}

\begin{proof}[\textbf{Proof of Theorem \ref{thm:pure:sing} and \ref{thm:pure:sing:reg}}] Combining the results in Theorem \ref{thm:pure:sing:exist} and Theorem \ref{thm:pure:sing:boundbeha}, we complete our proof.
\end{proof}

\begin{remark}
	Under the same assumption Theorem \ref{thm:pure:sing} and $f \in \delta^\gamma L^\infty(\Omega)$, $f \geq 0$ and with minor changes in the proof of Lemma \ref{est:green}, Proposition \ref{pro:app}, Theorem \ref{thm:pure:sing:exist} and Theorem \ref{thm:pure:sing:boundbeha}, we can prove the existence of a unique weak dual solution $ u \in \mathcal{A}_{\ast}(\Omega)$ of the perturbed problem
	\begin{equation}\label{sing:puref}
	 \left\{
	\begin{aligned}
	\mathbb{L} u
	&=\frac{1}{u^q} + f \quad &&\text{in } \Omega, \\
	u &> 0 &&\text{in } \Omega, \\
	u & = 0 &&\text{on } \partial\Omega \text{ or in } \Omega \text{ if applicable}.
	\end{aligned}
	\right.
	\end{equation}
\end{remark}
\section{Applications: Semilinear problems involving singular nonlinearities}\label{source:absorption}
In this section, we study the semilinear elliptic problems in the presence of singular nonlinearities and a source or absorption term via sub-super solution method.
\subsection{Source term} Let $u_*$ be the solution of \eqref{sing:pure} and put $v_0=u_*$. To study the problem \eqref{sing:f(v)}, first we prove the existence results for the following parameterized iterative scheme with parameter $\lambda$ and $\kappa$
	\begin{equation}\label{eq:iter} \tag{$Q_n$}
	\left\{
	\begin{aligned}
	\mathbb{L} v_n + \kappa v_n
	& =\frac{1}{v_n^q} + \lambda f(v_{n-1}) + \kappa v_{n-1} \quad &&\text{ in } \Omega, \\
	v_n & > 0 \quad &&\text{ in } \Omega, \\
	v_n & = 0 && \text{ on } \partial\Omega \text{ or in } \Omega^c \text{ if applicable},
	\end{aligned}
	\right.
	\end{equation}
for $n \geq 1$, via Schauder fixed point theorem which is required in the proof of Theorem \ref{thm:source}. The weak-dual solution of \eqref{eq:iter} is understood in the sense that
	\begin{equation} \label{notion:iter}
	\int_{\Omega} v_n \xi ~dx +  \kappa\int_{\Omega} v_n \mathbb{G}^\Omega[\xi] ~dx = \int_{\Omega} \left( \frac{1}{v_n^q} + \lambda f(v_{n-1}) + \kappa v_{n-1} \right) \mathbb{G}^\Omega[\xi] ~dx, \,\forall \, \xi \in \delta^\gamma L^\infty(\Omega).
	\end{equation}
\begin{proposition}\label{pro:iter:exis:1}
Assume \eqref{L1}--\eqref{L4}, \eqref{G1}--\eqref{G3}, \eqref{f1} hold  and $(q, \gamma) \in \mathbb{E}$. Then for any $n \in \mathbb{N}$, there exists a weak-dual solution $v_n \in \mathcal{A}_{\ast}(\Omega)$ of the iterative scheme \eqref{eq:iter}.
\end{proposition}
\begin{proof}
	Let $n=1$, $\epsilon>0$ and $\kappa>1$. Let $u_*$ be the solution of \eqref{sing:pure}. Put
	$$ M_\epsilon := \frac{1}{\varepsilon^q} + \| \lambda f(u_*) + \kappa u_* \|_{L^\infty(\Omega)}.
	$$
	For any function $v \in L^\infty(\Omega)$, put
	$$ A_{\epsilon}(v):=\{  x \in \Omega: - \kappa M_\epsilon \| \mathbb{G}^\Omega[1] \|_{L^\infty(\Omega)}^2 \leq v(x) \leq M_\varepsilon \| \mathbb{G}^\Omega[1] \|_{L^\infty(\Omega)}  \}.
	$$
	For any function $v \in L^\infty(\Omega)$,  there exists a unique weak-dual solution $\tilde v$ of the problem
	\begin{equation*}
	(P_i^\epsilon)^1 \left\{
	\begin{alignedat}{2}
	{} \mathbb{L} \tilde v + \kappa v^+{\bf 1}_{A_\epsilon(v)}
	& {}= \frac{1}{(v^++\epsilon)^q}  + \lambda f(v_0) + \kappa v_0,
	&& \quad\mbox{ in } \, \Omega, \\
	\tilde v & {}= 0
	&& \quad\mbox{ on } \partial\Omega\, \ \mbox{or in $\Omega^c$ if applicable}.
	\end{alignedat}
	\right.
	\end{equation*}
	The solution $\tilde v$ is represented as
	$$ \tilde v = \mathbb{G}^\Omega\left[\frac{1}{(v^+ +\epsilon)^q}  + \lambda f(v_0) + \kappa (v_0- v^+{\bf 1}_{A_\epsilon(v)} ) \right] \in L^\infty(\Omega).
	$$
Now, we define the solution map $\mathbb{S}_\epsilon:  L^\infty(\Omega) \to  L^\infty(\Omega)$ as
	$$ \mathbb{S}_\epsilon(v):= \mathbb{G}^\Omega\left[\frac{1}{(v^++\epsilon)^q}  + \lambda f(v_0) + \kappa (v_0- v^+{\bf 1}_{A_\epsilon(v)} ) \right], \quad v \in L^\infty(\Omega).$$
	On one hand, we see that
	\begin{equation} \label{T_1-bound1}
	\mathbb{S}_\epsilon(v)  \leq M_\epsilon \| \mathbb{G}^\Omega[1] \|_{L^\infty(\Omega)} \quad \text{for a. e. } x \in \Omega.
	\end{equation}
	On the other hand, we obtain
	\begin{equation} \label{T_1-bound2}
	\mathbb{S}_\epsilon(v) \geq - \kappa \mathbb{G}^\Omega[v^+{\bf 1}_{A_\epsilon(v)} ] \geq - \kappa M_\epsilon \| \mathbb{G}^\Omega[1] \|_{L^\infty(\Omega)}^2 \quad \text{for a.e. } x \in \Omega.
	\end{equation}
	Combining \eqref{T_1-bound1} and \eqref{T_1-bound2} implies
	$$ - \kappa M_\epsilon \| \mathbb{G}^\Omega[1] \|_{L^\infty(\Omega)}^2 \leq \mathbb{S}_\epsilon(v)(x) \leq M_\epsilon \| \mathbb{G}^\Omega[1] \|_{L^\infty(\Omega)} \quad \text{for a.e. } x \in \Omega.
	$$
	Put
	$$ \mathscr{E}_{\epsilon}:=\{  v \in L^\infty(\Omega): - \kappa M_\epsilon \| \mathbb{G}^\Omega[1] \|_{L^\infty(\Omega)}^2 \leq v(x) \leq  M_\epsilon \| \mathbb{G}^\Omega[1] \|_{L^\infty(\Omega)} \text{ for a.e. } x \in \Omega \}.
	$$
	Then $\mathscr{E}_{\epsilon}$ is a closed, convex subset of $L^\infty(\Omega)$ and $\mathbb{S}_\epsilon(\mathscr{E}_{\epsilon}) \subset \mathscr{E}_{\epsilon}$.\medskip
	
\noindent \textbf{Claim 1:} $\mathbb{S}$ is continuous.

Indeed, let $\{v_k\} \subset \mathscr{E}_{\epsilon}$ such that $v_k \to v$ in $L^\infty(\Omega)$. Put $\tilde v_k=\mathbb{S}_\epsilon(v_k)$ and $\tilde v=\mathbb{S}_\epsilon(v)$. Since $v_k \in \mathscr{E}_\epsilon$, it follows that $v \in \mathscr{E}_\epsilon$.  Moreover ${\bf 1}_{A_\epsilon(v_k)}= {\bf 1}_{A_\epsilon(v)} = 1$ a.e. in $\Omega$.
Then
	\begin{align*}
	|\tilde v_k- \tilde v| &\leq \mathbb{G}^\Omega \left[ \left| \frac{1}{(v_k^++\epsilon)^q}  - \frac{1}{(v^++\epsilon)^q} \right| \right] + \kappa \mathbb{G}^\Omega[|v_k^+- v^+|] \\
	&\leq \left(\frac{q}{\epsilon^{q+1}} +\kappa \right) \mathbb{G}^\Omega[|v_k^+-v^+|] \\
	&\leq \left(\frac{q}{\epsilon^{q+1}} +\kappa \right)\| v_k - v  \|_{L^\infty(\Omega)} \| \mathbb{G}^\Omega[1] \|_{L^\infty(\Omega)}.
	\end{align*}
	This implies
	$$ \|\tilde v_k - \tilde v  \|_{L^\infty(\Omega)} \leq \left(\frac{q}{\epsilon^{q+1}} +\kappa \right)\| v_k - v  \|_{L^\infty(\Omega)} \| \mathbb{G}^\Omega[1] \|_{L^\infty(\Omega)}.
	$$
	Since $v_k \to v$ in $L^\infty(\Omega)$, by letting $k \to \infty$ in the above estimate, we deduce that $\tilde v_k \to \tilde v$ in $L^\infty(\Omega)$. \medskip
	
\noindent \textbf{Claim 2:} $\mathbb{S}_\epsilon(\mathscr{E}_{\epsilon})$ is relatively compact. By Arzela-Ascoli theorem and $\mathbb{S}_\epsilon(\mathscr{E}_{\epsilon}) \subset \mathscr{E}_{\epsilon}$, it is enough to show that the set $\mathbb{S}_\epsilon(\mathscr{E}_{\epsilon})$ is equicontinuous.
	To this end, let $x \in \Omega$ and $\mu >0$. Therefore, there exists $\nu>0$ such that if $y \in B(x,\nu)$ then
	$$ \int_{\Omega}|G^{\Omega}(x,z)-G^{\Omega}(y,z)|dz < \mu.
	$$
	Now take $\tilde v \in \mathbb{S}_\epsilon(\mathscr{E}_{\epsilon})$. For $y \in B(x,\nu)$, then there exists $v \in \mathscr{E}_{\epsilon}$ such that $\tilde v=\mathbb{S}_\epsilon(v)$, we have
	\begin{equation*}
	\begin{aligned}
	&|\tilde v(x)- \tilde v(y)|\\
	&= \left| \mathbb{G}^\Omega\left[\frac{1}{(v^++\epsilon)^q}  + \lambda f(v_0) + \kappa (v_0- v^+{\bf 1}_{A_\epsilon(v)}) \right](x)  \right. \\
	& \quad \quad \left. - \mathbb{G}^\Omega\left[\frac{1}{(v^++\epsilon)^q}  + \lambda f(v_0) + \kappa (v_0- v^+{\bf 1}_{A_\epsilon(v)}) \right](y) \right| \\
	& \leq \left(\frac{1}{\epsilon^q} + \| \lambda f(v_0) + \kappa v_0  \|_{L^\infty(\Omega)} +  M_\epsilon \| \mathbb{G}^\Omega[1] \|_{L^\infty(\Omega)}  \right) \int_\Omega |G^{\Omega}(x,z)- G^{\Omega}(y,z)|dz \\
	&\leq  \left(\frac{1}{\epsilon^q} + \| \lambda f(v_0) + \kappa v_0  \|_{L^\infty(\Omega)} + M_\epsilon \| \mathbb{G}^\Omega[1] \|_{L^\infty(\Omega)} \right)\mu.
	\end{aligned}
	\end{equation*}
	Therefore $\mathbb{S}_\epsilon(\mathscr{E}_{\epsilon})$ is equicontinuous. By the Arzela-Ascoli theorem, $\mathbb{S}_\epsilon(\mathscr{E}_{\epsilon})$ is relatively compact.

	We infer from the Schauder fixed point theorem that there exists a fixed point $v_{\epsilon,1} \in \mathscr{E}_{\epsilon}$ of $\mathbb{S}_\epsilon$, namely
	$$  v_{\epsilon,1} = \mathbb{S}_\epsilon(v_{\epsilon,1}) = \mathbb{G}^\Omega\left[\frac{1}{(v_{\epsilon,1}^++\epsilon)^q}  + \lambda f(v_0) + \kappa (v_0- v_{\epsilon,1}^+) \right].
	$$
	From Lemma \ref{IBP}, we see that  $v_{\epsilon,1}$ is the unique weak-dual solution of
	\begin{equation*}
	\left\{
	\begin{alignedat}{2}
	{} \mathbb{L} v_{\epsilon,1} + \kappa v_{\epsilon,1}^+
	& {}=\frac{1}{(v_{\epsilon,1}^++ \epsilon)^q} + \lambda f(v_{0}) + \kappa v_{0} \
	&& \quad\mbox{ in } \, \Omega, \\
	v_{\epsilon,1} & {}= 0
	&& \quad\mbox{ on } \partial\Omega\, \ \mbox{or in $\Omega^c$ if applicable}.
	\end{alignedat}
	\right.
	\end{equation*}

	To prove the positivity and boundary behavior of the weak-dual solution $v_{\epsilon,1}$, we construct a positive subsolution and supersolution of the above problem. We take $u_\epsilon$, which is the weak-dual solution of \eqref{sing:approx}, as a weak-dual subsolution of the above problem. We see that
	$$ u_\epsilon - v_{\epsilon,1} = \mathbb{G}^\Omega\left[\frac{1}{(u_{\epsilon}+\epsilon)^q}  - \frac{1}{(v_{\epsilon,1}^++\epsilon)^q}  - \lambda f(v_0) - \kappa (v_0- v_{\epsilon,1}^+) \right].
	$$
	By using Kato type inequality \eqref{kato+}, we obtain that, for any $0 \leq \xi \in C_c^\infty(\Omega)$ such that $\mathbb{G}^\Omega[\xi] \geq 0$,
	\begin{align*} \int_{\Omega}(u_\epsilon - v_{\epsilon,1})^+ \xi ~dx &\leq \int_{\{ u_\epsilon \geq v_{\epsilon,1} \} }\left[\frac{1}{(u_{\epsilon}+\epsilon)^q}  - \frac{1}{(v_{\epsilon,1}^++\epsilon)^q}  - \lambda f(v_0) - \kappa (v_0- v_{\epsilon,1}^+) \right] \mathbb{G}^\Omega[\xi]~dx \\
	&\leq 0.
	\end{align*}
The second inequality follows from the fact that $v_0 \geq u_\epsilon$. Therefore $u_\epsilon \leq v_{\epsilon,1}$. In particular $v_{\epsilon,1} \geq 0$. Next, put $V_\epsilon = u_\epsilon + M \mathbb{G}^\Omega[1]$ for some $M>0$ which will be determined later. Then
$$ V_\epsilon = \mathbb{G}^\Omega\left[\frac{1}{(u_{\epsilon}+\epsilon)^q}+M \right].$$
It follows that
	$$ v_{\epsilon,1} - V_{\epsilon} = \mathbb{G}^\Omega\left[  \frac{1}{(v_{\epsilon,1}+\epsilon)^q} - \frac{1}{(u_{\epsilon}+\epsilon)^q} + \lambda f(v_0) + \kappa (v_0- v_{\epsilon,1}) - M \right] .
	$$
	Now by choosing $M = \|\lambda f(v_0) + \kappa v_0\|_{L^\infty(\Omega)}$ and applying the Kato inequality \eqref{kato+}, we obtain
	\begin{equation*}
	\begin{aligned} \int_{\Omega} (&v_{\epsilon,1} - V_{\epsilon})^+ \xi ~dx \\
	&\leq \int_{\{ v_{\epsilon,1} \geq V_{\epsilon} \} }\left[  \frac{1}{(v_{\epsilon,1}+\epsilon)^q} - \frac{1}{(u_{\epsilon}+\epsilon)^q} + \lambda f(v_0) + \kappa (v_0- v_{\epsilon,1}) - M \right] \mathbb{G}^\Omega[\xi]~dx \leq 0.
	\end{aligned}
	\end{equation*}
	Hence $v_{\epsilon,1} \leq V_\epsilon$  in $\Omega$. Next, by using again Kato's inequality, we obtain, for $\epsilon > \epsilon'$,
	\begin{align*} \int_{\Omega}(v_{\epsilon,1} -v_{\epsilon',1})^+ \xi~dx &\leq \int_{\{ v_{\epsilon,1} \geq v_{\epsilon',1} \} }\left[  \frac{1}{(v_{\epsilon,1}+\epsilon)^q} - \frac{1}{(v_{\epsilon',1}+\epsilon')^q} - \kappa (v_{\epsilon,1}- v_{\epsilon',1})  \right] \mathbb{G}^\Omega[\xi]~dx  \\
	&\leq 0.
	\end{align*}
	It follows that $v_{\epsilon,1} \leq v_{\epsilon',1}$. Therefore $\epsilon \mapsto v_{\epsilon,1}$ is decreasing and
	  $0<  v_\epsilon \leq v_{\epsilon,1} \leq V_\epsilon$ in $\Omega$ for every $\epsilon >0$.  Put $v_1:= \lim_{\epsilon \to 0} v_{\epsilon,1}$.
	  We notice that
	  \begin{equation} \label{eq:notion}
	  \int_{\Omega} v_{\epsilon,1} \xi ~dx = \int_{\Omega} \left[\frac{1}{(v_{\epsilon,1}+\epsilon)^q}  + \lambda f(v_0) + \kappa (v_0- v_{\epsilon,1}) \right] \mathbb{G}^\Omega[\xi] ~dx \quad \forall \ \xi \in \delta^\gamma L^\infty(\Omega).
	  \end{equation}
	  By using estimate $u_\epsilon \leq v_{\epsilon,1} \leq V_\epsilon$ in $\Omega$, $u_\epsilon, V_\epsilon \in \mathcal{A}_{\ast}(\Omega)$ and by letting $\epsilon \to 0$, we obtain that $v_1$ satisfies \eqref{eq:iter} for $n=1$ and $v_1 \in \mathcal{A}_{\ast}(\Omega).$ The remaining proof can be done by induction.
\end{proof}
Now, we prove our main result on problem \eqref{sing:f(v)} involving the source term $\lambda f(u)$.

\begin{proof}[\textbf{Proof of Theorem \ref{thm:source}}]
	Let $u_* \in \mathcal{A}_{\ast}(\Omega)$ be the weak-dual solution of \eqref{sing:pure} and put
	$$v^0:= u_* +  U^0$$
	where $U^0= \Lambda \mathbb{G}^\Omega[1]$ is the weak-dual solution $\mathbb{L} U^0 = \Lambda \ \text{in} \ \Omega$. By choosing $\Lambda$ such that  $\Lambda \geq \lambda \|f(u_* + \Lambda \mathbb{G}^\Omega[1])\|_{L^\infty(\Omega)}$ due to assumption \eqref{f2}, we deduce that $u_*$ and $v^0$ are sub weak-dual solution and super weak-dual solution of \eqref{sing:f(v)} respectively.
	
	For $ n \geq 1$,  let $v_n \in \mathcal{A}_{\ast}(\Omega) \cap L^{1}(\Omega, \delta^\gamma)$ be the weak-dual solution of the iterative scheme \eqref{eq:iter}.	The existence of the weak-dual solution $v_n \in \mathcal{A}_{\ast}(\Omega)$ is guaranteed by Proposition \ref{pro:iter:exis:1}. Now, we claim that
	\begin{equation} \label{monotone-un} 0< u_* \leq  v_n \leq v^0 \quad \text{ for all } n \geq 1.
	\end{equation}
	Let $0 \not \equiv \xi \in C_c^\infty(\Omega)$ such that $\xi \geq 0 $ and $ \mathbb{G}^\Omega[\xi] \geq 0 $. Then using the definition of $u_*$ and Kato type inequality \eqref{kato+}, we obtain
	\begin{equation} \label{u1>u0}
	0 \leq \int_\Omega (u_*- v_1)^+ \xi ~dx \leq \int_{ \{u_* \geq v_1\} }  \left( \frac{1}{u_*^q}- \frac{1}{v_1^q}- \lambda f(u_*) + \kappa(v_1-u_*) \right) \mathbb{G}^\Omega[\xi]  ~dx \leq 0.
	\end{equation}
	This implies $(u_*-v_1)^+ =0$ in $\Omega$, whence $u_* \leq v_1$ in $\Omega$. 
	Then by using assumption \eqref{f3}, Kato type inequality \eqref{kato+}, for any $0 \not \equiv \xi \in C_c^\infty(\Omega)$ and $\xi, \mathbb{G}^\Omega[\xi] \geq 0$, we obtain
	\begin{equation} \label{u1<u^0}
	\begin{split}
	0 &\leq \int_\Omega (v_1- v^0)^+ \xi ~dx
	\leq \int_{ \{ v_1 \geq v^0 \}} \left( \frac{1}{v_1^q}- \frac{1}{u_*^q}+ \lambda f(u_*)- \Lambda + \kappa (u_0-u_1) \right) \mathbb{G}^\Omega[\xi]  ~dx\\
	& \leq \int_{ \{ v_1 \geq v^0 \}}  \left( \frac{1}{v_1^q}- \frac{1}{(v^0)^q}+ \lambda ( f(u_*)-f(v^0)) +  \kappa (u_*-v^0) + \kappa(v^0-v_1) \right) \mathbb{G}^\Omega[\xi]  ~dx \leq 0.
	\end{split}
	\end{equation}
	This gives $(v_1-v^0)^+ = 0$ in $\Omega$, whence $v_1 \leq v^0$ in $\Omega$. By using the argument similar to the one leading to \eqref{u1>u0} and \eqref{u1<u^0} for any $n \in \mathbb{N}$, we obtain \eqref{monotone-un}.
	
	Put $v:= \lim_{n \to \infty} v_n$. By using \eqref{monotone-un} and letting $n \to \infty$ in \eqref{notion:iter}, we obtain $v \in \mathcal{A}_{\ast}(\Omega)$ and
	$$\int_{\Omega} v \xi ~dx = \int_{\Omega} \frac{1}{v^q} \mathbb{G}^\Omega[\xi] ~dx + \lambda \int_{\Omega} f(v) \mathbb{G}^\Omega[\xi] ~dx, \quad \forall \, \xi \in \delta^\gamma L^\infty(\Omega).$$
	This means $v$ is a weak-dual solution of \eqref{sing:f(v)}.
\end{proof}
\subsection{Absorption term}
In this part, we focus on the semilinear elliptic problem \eqref{sing+absorption} involving singular nonlinearities and absorption term. Let $u_* \in \mathcal{A}_{\ast}(\Omega)$ be the solution of \eqref{sing:pure} and put $w_0=u_*$. We start by studying the existence result of the following paramterized iterative scheme involving a positive parameter $\mu$ \begin{equation}\label{eq:iter-abs} \tag{$\widetilde Q_n$}
	\left\{
	\begin{aligned}
	\mathbb{L} w_n + \mu w_n + g(w_{n-1})
	& =\frac{1}{w_n^q} + \mu w_{n-1} \quad &&\text{ in } \Omega, \\
	w_n & > 0 \quad &&\text{ in } \Omega, \\
	w_n & = 0 && \text{ on } \partial\Omega \text{ or in } \Omega^c \text{ if applicable},
	\end{aligned}
	\right.
	\end{equation}
for $n \geq 1$,	where weak-dual solution of \eqref{eq:iter-abs} is understood in the sense that
	\begin{equation}\label{notion:iter:ab}
	\int_{\Omega} w_n \xi ~dx +  \int_{\Omega} (\mu w_n + g(w_{n-1}))\ \mathbb{G}^\Omega[\xi] ~dx = \int_{\Omega} \left( \frac{1}{w_n^q} + \mu w_{n-1} \right) \mathbb{G}^\Omega[\xi] ~dx, \quad \forall \, \xi \in \delta^\gamma L^\infty(\Omega).
	\end{equation}
 \begin{proposition}\label{pro:iter:exis:2} Assume \eqref{L1}--\eqref{L4}, \eqref{G1}--\eqref{G3}, \eqref{g1}, \eqref{g3} hold and $(q, \gamma) \in \mathbb{E}$. Then, for any $n \in \mathbb{N}$, there exists a weak dual solution $w_n \in \mathcal{A}_{\ast}(\Omega)$ of the iterative scheme \eqref{eq:iter-abs}.
\end{proposition}
\begin{proof}
	Let $n=1$,  $\epsilon>0$ and $\mu>0$.  Let $u_* \in \mathcal{A}_{\ast}(\Omega)$ be the solution of \eqref{sing:pure} and put $w_0=u_*$. Denote
	$$ N_\epsilon := \frac{1}{\varepsilon^q} + \mu \|w_0 \|_{L^\infty(\Omega)}. $$
	For any given function $w \in L^\infty(\Omega)$, set
	$$ B_{\epsilon}(w):=\left\{  x \in \Omega: - \mu N_\epsilon \| \mathbb{G}^\Omega[1] \|_{L^\infty(\Omega)}^2 - \|\mathbb{G}^\Omega[g(w_0)]\|_{L^\infty(\Omega)} \leq w(x) \leq N_\varepsilon \| \mathbb{G}^\Omega[1] \|_{L^\infty(\Omega)}  \right\}.$$
Using assumption \eqref{g1}, \eqref{g3}, Lemma \ref{est:green} and the fact that $w_0 \in \mathcal{A}_{\ast}(\Omega)$, we notice that $g(w_0) \in L^1(\Omega, \delta^\gamma)$ and $\mathbb{G}^\Omega[g(w_0)] \in L^\infty(\Omega).$ Since $\mathbb{G}^\Omega$ maps $L^\infty(\Omega)$ to $L^\infty(\Omega)$ (see, \cite[Theorem 2.1]{AbaGomVaz_2019}), for any $w \in L^\infty(\Omega)$ there exists a unique weak-dual solution $\tilde w$ of the following problem
\begin{equation*}
	\left\{
	\begin{alignedat}{2}
		{} \mathbb{L} \tilde w
		& {}= \frac{1}{(w^++\epsilon)^q}  - g(w_0) + \mu (w_0- w^+ {\bf{1}}_{B_\epsilon(w)}),
		&& \quad\mbox{ in } \, \Omega, \\
		\tilde w & {}= 0
		&& \quad\mbox{ on } \partial\Omega\, \ \mbox{or in $\Omega^c$ if applicable}.
	\end{alignedat}
	\right.
\end{equation*}
The solution $\tilde w$ is represented as
$$\tilde w= \mathbb{G}^\Omega\left[\frac{1}{(w^++\epsilon)^q}  - g(w_0) + \mu (w_0- w^+ {\bf{1}}_{B_\epsilon(w)})\right] \in L^\infty(\Omega). $$

\noindent Now, we define the solution map $\mathbb{P}_\epsilon:  L^\infty(\Omega) \to  L^\infty(\Omega)$ as
$$ \mathbb{P}_\epsilon(w):=\mathbb{G}^\Omega\left[\frac{1}{(w^++\epsilon)^q}  - g(w_0) + \mu (w_0- w^+ {\bf{1}}_{B_\epsilon(w)})\right].$$
Using the same arguments as in Proposition \ref{pro:iter:exis:1}, we obtain, for any $w \in L^\infty(\Omega)$,
\begin{equation*} \label{P_1-bound1}
- \mu N_\epsilon \| \mathbb{G}^\Omega[1] \|_{L^\infty(\Omega)}^2 - \|\mathbb{G}^\Omega[g(w_0)]\|_{L^\infty(\Omega)}	\leq \mathbb{P}_\epsilon(w)(x)  \leq N_\epsilon \| \mathbb{G}^\Omega[1] \|_{L^\infty(\Omega)} \quad \text{for a.e. } x \in \Omega.
	\end{equation*}
Set
\begin{align*}
 \mathscr{F}_{\epsilon}:=\{  v \in L^\infty(\Omega):- \mu N_\epsilon \| \mathbb{G}^\Omega[1] \|_{L^\infty(\Omega)}^2 &- \|\mathbb{G}^\Omega[g(w_0)]\|_{L^\infty(\Omega)} \leq v(x) \\
    &\leq  N_\epsilon \| \mathbb{G}^\Omega[1] \|_{L^\infty(\Omega)} \text{ for a.e. } x \in \Omega \}.
\end{align*}
Then $\mathscr{F}_{\epsilon}$ is a closed, convex subset of $L^\infty(\Omega)$ and  $\mathbb{P}_\epsilon(\mathscr{F}_\epsilon) \subset \mathscr{F}_\epsilon$.
Now, by adopting the same arguments of continuity and compactness by Schauder fixed point theorem as in proof of Proposition \ref{pro:iter:exis:1}, there exists at least one fixed point $w_{\epsilon,1} \in \mathscr{F}_\epsilon$ of the map $\mathbb{P}_\epsilon$ such that $\mathbb{P}_\epsilon(w_{\epsilon,1})=w_{\epsilon,1}$. Finally, from Lemma \ref{Hopf}, we obtain $$w_{\epsilon,1}= \mathbb{G}^\Omega\left[\frac{1}{(w_{\epsilon,1}^++\epsilon)^q} + g(w_0) + \mu (w_0-w_{\epsilon,1}^+)\right],$$
hence $w_{\epsilon,1}$ satisfies
	\begin{equation*}
	\left\{
	\begin{alignedat}{2}
	{} \mathbb{L} w_{\epsilon,1} + \mu w_{\epsilon,1}^+ + g(w_0)
	& {}=\frac{1}{(w_{\epsilon,1}^++ \epsilon)^q} +  \mu w_0, \
	&& \quad\mbox{ in } \, \Omega, \\
	w_{\epsilon,1} & {}= 0
	&& \quad\mbox{ on } \partial\Omega\, \ \mbox{or in $\Omega^c$ if applicable}.
	\end{alignedat}
	\right.
	\end{equation*}
	
To prove the positivity and boundary behavior of the weak-dual solution $w_{\epsilon,1}$, we construct a positive sub weak-dual  solution and super weak-dual  solution of the above problem. For super weak-dual solution, define $\overline{w}:= w_0$ such that  $$\mathbb{L} \overline{w} + \mu \overline{w}^+ + g(w_0) \geq \frac{1}{(\overline{w}^+ + \epsilon)^q} + \mu w_0 \ \text{in} \ \Omega.$$
	For sub weak-dual solution, define $\underline{w}= k u_{\epsilon_k} \leq k w_0$ where $\epsilon_k:= \epsilon k^{-1}$, $u_{\epsilon_k}$ is the weak-dual solution of problem $(P_{\epsilon_k})$ and $k$ is chosen small enough such that
	$$k^q( \|u_{\epsilon_k} + \epsilon_k\|_{L^\infty(\Omega)}^q \|g(w_0)\|_{L^\infty(\Omega)} + k) \leq 1.$$ Then
	\begin{equation}
	\begin{split}
	\mathbb{L} \underline{w} + \mu \underline{w}^+ + g(w_0) \leq \frac{k}{(u^+_{\epsilon_k} + \epsilon_k)^q} + \mu \underline{w}^+ + g(w_0) &\leq \frac{k^{q+1}}{(\underline{w}^+ + \epsilon)^q} + \mu w_0 + \frac{1-k^{q+1}}{(\underline{w}^+ + \epsilon)^q} \\
	& \leq \frac{1}{(\underline{w}^+ + \epsilon)^q} + \mu w_0 \ \text{in} \ \Omega.
	\end{split}
	\end{equation}
	Then from Kato type inequality \eqref{kato+}, we obtain $0< \underline{w} \leq w_{\epsilon,1} \leq w_0= \overline{w}$. Now by defining $w_1:= \lim_{\epsilon \to 0} w_{\epsilon,1}$ and by using the boundary behavior of $u_\epsilon, w_0$ and by passing limits $\epsilon \to 0$, we obtain $w_1$ satisfies \eqref{eq:iter-abs} for $n=1$ and $w_1 \in \mathcal{A}_{\ast}(\Omega)$. The remaining proof can be done by induction and we omit it.
	\end{proof}

\begin{proof}[\textbf{Proof of Theorem \ref{thm:absor}}]
Let $u_* \in \mathcal{A}_{\ast}(\Omega)$ is the weak dual solution of \eqref{sing:pure} which act as a super weak-dual solution of the problem \eqref{sing+absorption}. Put $w_0=u_*$.

Now, to construct the weak-dual subsolution of the problem \eqref{sing+absorption}, because of $(g_q)$ we choose a constant $\ell>0$ small enough such that
	\[
	\ell^q (\|w_0\|_{L^\infty(\Omega)}^q \|g(\ell w_0)\|_{L^\infty(\Omega)} + \ell) \leq 1.
	\]
	Put $\underline w_0:= \ell w_0$ then $\underline w_0$ is a sub weak-dual solution of problem \eqref{sing+absorption} since it satisfies
	\begin{equation}
	\mathbb{L} \underline w_0 + g(\underline w_0)= \frac{\ell^{q+1}}{\underline w_0^q} + g(\ell \underline w_0) \leq \frac{1-\ell^{q+1}}{\underline w_0^q} + \frac{\ell^{q+1}}{\underline w_0^q} = \frac{1}{\underline w_0^q}.
	\end{equation}
	For $ n \geq 1$,  let $w_n \in \mathcal{A}_{\ast}(\Omega) \cap L^{1}(\Omega, \delta^\gamma)$ be the weak-dual solution of the iterative scheme \eqref{eq:iter-abs}.  The existence of the weak-dual solution $w_n \in \mathcal{A}_{\ast}(\Omega)$ is guaranteed by Proposition \ref{pro:iter:exis:2}.
	
	We claim that
	\begin{equation} \label{monotone-wn} 0< \underline w_0 \leq w_n \leq w_{n-1} \leq w_0 \ \text{ for all} \ n \geq 1.
	\end{equation}
	To this end, we choose $\mu$ large enough such that $t \to \mu t - g(t)$ is increasing in the interval $(0, \|w_0\|_{L^\infty(\Omega)}]$ due to assumption \eqref{g2}.
	Now by using the definition of $\underline w_0, w_0$ and Kato's inequality for $0 \not \equiv \xi \in C_c^\infty(\Omega)$ such that $\xi \geq 0 $ and $ \mathbb{G}^\Omega[\xi] \geq 0 $, we obtain
	\begin{equation}
	\begin{split}
	0 &\leq  \int_\Omega (\underline w_0-w_1)^+ \xi ~dx \\
	& \leq \int_{ \{\underline w_0 \geq w_1  \}} \left( \frac{l^{q+1}}{v_0^q}- \frac{1}{w_1^q}+ g(\ell w_0) + (g(w_0)-g(\underline w_0)) + \mu(w_1-w_0) \right) \mathbb{G}^\Omega[\xi]  ~dx \\
	& \leq \int_{ \{\underline w_0 \geq w_1  \}}  \left( \frac{1}{\underline w_0^q}- \frac{1}{w_1^q}+ (g(w_0)-g(\underline w_0)) +  \mu (\underline w_0-w_0) + \mu (w_1- \underline w_0) \right) \mathbb{G}^\Omega[\xi]  ~dx \leq 0
	\end{split}
	\end{equation}
	and
	\begin{equation}
	0 \leq \int_\Omega (w_1 - w_0)^+ \xi ~dx \leq \int_{ { \{w_1 \geq w_0  \}}}  \left( \frac{1}{w_1^q}- \frac{1}{w_0^q}- g(w_1) + \mu (w_0-w_1) \right) \mathbb{G}^\Omega[\xi]  ~dx \leq 0.
	\end{equation}
	The above estimate implies $\underline w_0 \leq w_1 \leq w_0$ in $\Omega$. By using a similar argument, we obtain \eqref{monotone-wn}.
	
	Next put $w:= \lim_{n \to \infty} w_n$. By using \eqref{monotone-wn} and by letting $n \to \infty$ in \eqref{notion:iter:ab}, we obtain $w \in \mathcal{A}_{\ast}(\Omega)$ and
	$$\int_{\Omega} w \xi ~dx + \int_{\Omega} g(w) \mathbb{G}^\Omega[\xi] ~dx = \int_{\Omega} \frac{1}{w^q} \mathbb{G}^\Omega[\xi] ~dx  \quad \forall \ \xi \in \delta^\gamma L^\infty(\Omega).$$
	This means that $w$ is a solution of \eqref{sing+absorption}. The proof is complete.
\end{proof}

\section{Appendix}\label{example:singpotential}
In this section, we discuss the Laplace operator with Hardy potential $L_\mu$ which is strongly singular on the boundary
\[
L_\mu := -\Delta - \frac{\mu}{\delta^2},
\]
where $\mu \in \mathbb{R}$ is a parameter and $\delta(x)=\dist(x,\partial \Omega)$. This operator has been investigated extensively in the literature (see \cite{Filippas_Moschini_tertikas_2007,Bandle_Moroz_Reichel_2008, Gkikas_Veron_2015, Gkikas_Nguyen_2019, Marcus_Nguyen_2017}) and acts as a special case of Schr\"odinger operators. The parameter $\mu$ is imposed to satisfy $\mu \in [0, C_H(\Omega))$ where $C_H(\Omega)$ is the best constant in the Hardy's inequality given by
$$
C_{H}(\Omega):= \inf_{\varphi \in H_0^1(\Omega) \setminus \{0\}} \frac{\int_{\Omega} |\nabla \varphi|^2 ~dx}{\int_{\Omega} |\varphi/\delta|^2 ~dx}.
$$

We will show that the operator $L_\mu$ satisfies assumptions mentioned in Subsection \ref{Main assumptions}. It was showed in \cite[Section 9]{HuNg_source} that assumptions \eqref{L1}--\eqref{L3} and \eqref{G1}--\eqref{G2} are fulfilled for $L_\mu$. In particular, in this case, \eqref{G2} holds for $s=1$ and $\gamma = \frac{1}{2} + \sqrt{\frac{1}{4} - \mu}$ and assumption \eqref{L3} becomes $\mathbb{H}(\Omega)=H_0^1(\Omega)$. We also have
\begin{equation} \label{innerH01}
\langle  u, v\rangle_{H_0^1(\Omega)} = \int_{\Omega} \nabla u \cdot \nabla v ~dx - \mu \int_{\Omega} \frac{u(x)v(x)}{\delta(x)^2} ~dx, \quad \forall u,v \in H_0^1(\Omega).
\end{equation}

Next, assumption \eqref{G3} holds true by using a similar argument as in the case of the sum of two RFLs in Subsection \ref{subsec:examples} together with the continuity and small and long time estimates of the heat kernel in \cite[(1.2) and (1.3)]{Filippas_Moschini_tertikas_2007}.

 We notice that assumption \eqref{L4} is only used in proving the Kato type inequality in Proposition \ref{Kato} and due to the presence of the potential, the operator $L_\mu$ does not satisfy the assumption \eqref{L4}. Now, in order to prove Kato type inequality and ensure that our results remains true for this operator, we replace assumption \eqref{L4} by a stronger assumption involving the singular potential as follows
\begin{enumerate}[label=(L4)$_\mu$,ref=(L4)$_\mu$]
	\item \label{L4-new} 	For any $u,v\in \mathbb{H}(\Omega) \cap L^\infty(\Omega)$ with $v \geq 0$ and any convex function $p \in C^{1,1}(\mathbb{R})$ with $p(0)=p'(0)=0$, there holds
	\[
	\left\langle p(u), v \right\rangle_{\mathbb{H}(\Omega)} \leq \left\langle  u,  p'(u)  v \right\rangle_{\mathbb{H}(\Omega)} + \mu \left\langle p'(u) u - p(u), \frac{v}{\delta^2}\right\rangle_{L^2(\Omega)}.
	\]
\end{enumerate}

We will show below that the operator $L_\mu$ satisfies assumption \ref{L4-new}. Take $u,v\in \mathbb{H}(\Omega) \cap L^\infty(\Omega)$ with $v \geq 0$ and a convex function $p \in C^{1,1}(\mathbb{R})$ with $p(0)=p'(0)=0$. Then by Lemma \ref{lip-prop}, $p(u), p'(u)v \in H_0^1(\Omega)$. Moreover, since $p$ is a convex function, it is twice differentiable a.e. Then, by applying \eqref{innerH01} for $u$ and $p'(u)v$, we get
\[
\begin{split}
    \left\langle  u,  p'(u)  v \right\rangle_{H_0^1(\Omega)} & + \mu \left\langle p'(u) u - p(u), \frac{v}{\delta^2}\right\rangle_{L^2(\Omega)} \\
    & = \int_{\Omega} |\nabla u|^2 p''(u) v ~dx + \int_{\Omega} p'(u) \nabla u \cdot \nabla v ~dx -\mu  \int_{\Omega} p(u) \frac{v}{\delta^2} ~dx  \\
    & \geq \int_{\Omega} \nabla (p(u)) \cdot \nabla v ~dx - \mu  \int_{\Omega} p(u) \frac{v}{\delta^2} ~dx  = 	\left\langle p(u), v \right\rangle_{H_0^1(\Omega)}.
\end{split}
\]
It means that \ref{L4-new} is fulfilled.

With suitable changes in the proof of Proposition \ref{Kato} in light of assumption \ref{L4-new}, we can prove the required Kato type inequality. For the sake of completeness, we give a sketch of the proof.

First assume that $f \in C_c^\infty(\Omega)$, $u=\mathbb{G}^{\Omega}[f]$, $p \in C^{1,1}(\mathbb{R})$ is a convex function such that $p(0)=p'(0)=0$ and $|p'| \leq 1$ and $\xi \in \delta L^\infty(\Omega)$ such that $\mathbb{G}^{\Omega}[\xi] \geq 0$. By  \eqref{equi:notion} and \ref{L4-new}, we obtain
 \begin{equation} \label{KT-2-n}
 \begin{split}
    \int_{\Omega}  f p'(u) \mathbb{G}^\Omega[\xi] ~dx  &=
     \left\langle  \mathbb{G}^{\Omega}[f],  p'(u)  \mathbb{G}^{\Omega}[\xi] \right\rangle_{H_0^1(\Omega)} = \left\langle  u,  p'(u)  \mathbb{G}^{\Omega}[\xi] \right\rangle_{H_0^1(\Omega)} \\
     &   \geq  \left\langle p(u), \mathbb{G}^{\Omega}[\xi]  \right\rangle_{H_0^1(\Omega)} -\mu \left\langle p'(u) u - p(u) , \frac{\mathbb{G}^{\Omega}[\xi]}{\delta^2} \right\rangle_{L^2(\Omega)} \\
     &=\int_{\Omega} p(u) \xi ~dx -\mu \int_{\Omega} (p'(u) u - p(u)) \frac{\mathbb{G}^{\Omega}[\xi]}{\delta^2} ~dx.
 \end{split}
 \end{equation}
The above inequality holds for any $f \in L^1(\Omega, \delta^\gamma)$ by repeating the approximation arguments in \textbf{Claim} 2 of Proposition \ref{Kato}.

Now, consider the sequence  $\{p_k\}_{k \in \mathbb{N}}$ defined in \eqref{eq:sequencepk}.
Then for every $k \in \mathbb{N}$, $p_k \in C^{1,1}(\mathbb{R})$ is convex, $p_k(0) = (p_k)'(0) = 0$ and $|(p_k)'| \le 1$.  Hence, employing \eqref{KT-2-n} with $p = p_k$, we have, for any $\xi \in \delta  L^{\infty}(\Omega), \mathbb{G}^{\Omega} [\xi]\ge 0$,
\begin{equation} \label{pku-new}
	\int_{\Omega} p_k(u) \xi ~dx - \mu \int_{\Omega} g_k(u) \frac{\mathbb{G}^{\Omega}[\xi]}{\delta^2} ~dx \le \int_{\Omega} f (p_k)'(u) \mathbb{G}^{\Omega}[\xi] ~dx
\end{equation}
where $g_k(t):= p_k'(t) t -p_k(t).$ Notice that $p_k(t) \to |t|$ and $(p_k)'(t) \to \sign (t)$ as $k \to +\infty$. Then, by using the Lebesgue dominated convergence theorem, we obtain
\[
\int_{\Omega} p_k(u) \xi ~dx \to \int_{\Omega} |u| \xi ~dx \quad \text{and} \quad \int_{\Omega} (p_k)'(u) \mathbb{G}^{\Omega}[\xi] f ~dx \to \int_{\Omega} \sign^+(u) \mathbb{G}^{\Omega}[\xi] f ~dx.
\]
Now, in order to prove the required inequality \eqref{kato||}, it is enough to claim that
\[
\int_{\Omega} g_k(u) \frac{\mathbb{G}^{\Omega}[\xi]}{\delta^2} ~dx \to 0 \ \text{as} \ k \to \infty.
\]
Let $\epsilon >0$ be given. Since $\frac{\mathbb{G}^{\Omega}[\xi]}{\delta} \in L^2(\Omega)$, there exists a $\eta = \eta(\epsilon) >0$ such that
\[
\text{if} \ A \ \text{is a measurable subset of} \ \Omega \ \text{with}\ |A| < \eta \text{ then } \int_{A} \left(\frac{\mathbb{G}^{\Omega}[\xi]}{\delta}\right)^2 ~dx < \epsilon.
\]
It is easy to observe that $g_k(u) \in H_0^1(\Omega)$ and $g'_k(u) \leq 1$ for all $k$. Let $\epsilon_1 >0$. Then, by using the Hardy inequality, and taking $\epsilon < \frac{\epsilon_1^2}{\|u\|^2_{H_0^1(\Omega)}}$ and a measurable set $A \subset \Omega$ such that $|A| < \eta$, we obtain
\[
\int_{A} g_k(u) \frac{\mathbb{G}^{\Omega}[\xi]}{\delta^2} ~dx \leq \left\|\frac{g_k(u)}{\delta}\right\|_{L^2(\Omega)} \left\|\frac{\mathbb{G}^{\Omega}[\xi]}{\delta}\right\|_{L^2(A)} \leq \epsilon^\frac{1}{2} \|g_k(u)\|_{H_0^1(\Omega)} \leq \epsilon^\frac{1}{2} \|u\|_{H_0^1(\Omega)} < \epsilon_1.
\]
Hence, the sequence $\{g_k(u) \mathbb{G}^{\Omega}[\xi] \delta^{-2}\}_{k \in \mathbb{N}}$ is equi-integrable. Finally, by using the fact that  $g_k(t) \to 0$ as $k \to \infty$ for every $t \in \mathbb{R}$ and employing the Vitaly convergence theorem, we get the desired result.
Inequality \eqref{kato+} follows by the same argument as in the proof of Proposition \ref{Kato}.

\subsection*{Acknowledgments}
The research of Vicen\c{t}iu D. R\u{a}dulescu was supported by a grant of
the Romanian Ministry of Research, Innovation and Digitization, CNCS-UEFISCDI,
project number PCE 137/2021, within PNCDI III. The research of Rakesh Arora and Phuoc-Tai Nguyen were supported by by Czech Science Foundation, Project GA22-17403S.

\bibliographystyle{siam}

\end{document}